\documentclass[12pt,a4paper]{amsart}

\usepackage{amsmath,amsthm,amsfonts}
\usepackage{amssymb}
\usepackage{hyperref}
\usepackage{graphicx}
\usepackage{tikz}
\usepackage{caption}
\usepackage{subcaption}
\usepackage{bbm}
\usetikzlibrary{positioning}

\newtheorem{thm}{Theorem}[section]
\newtheorem{lemma}[thm]{Lemma}
\newtheorem{prop}[thm]{Proposition}
\newtheorem{cor}[thm]{Corollary}
\newtheorem{defn}[thm]{Definition}

\newtheorem{example}[thm]{Example}
\newtheorem{remark}[thm]{Remark}

\newcommand{\N}{\mathbb{N}}

\newcommand{\R}{\mathbb{R}}

\newcommand{\Z}{\mathbb{Z}}

\newcommand{\NN}{\mathcal{N}}

\newcommand{\supp}{\operatorname{supp}}

\newcommand{\D}{\Delta}
\newcommand{\1}{\mathbbm{1}}

\def\K{\mathcal{K}}

\newcommand{\udots}{\reflectbox{$\ddots$}}

\begin{document}

\title{Curvatures, graph products and Ricci flatness}

\date{\today}


\author[Cushing]{David Cushing}
\address{D. Cushing, Department of Mathematical Sciences, Durham University, Durham DH1 3LE, United Kingdom}
\email{davidcushing1024@gmail.com}

\author[Kamtue]{Supanat Kamtue}
\address{S. Kamtue, Department of Mathematical Sciences, Durham University, Durham DH1 3LE, United Kingdom}
\email{supanat.kamtue@durham.ac.uk}

\author[Kangaslampi]{Riikka Kangaslampi}
\address{R. Kangaslampi, Unit of Computing Sciences, Tampere University 33014, Tampere 33014, Finland}
\email{riikka.kangaslampi@tuni.fi }

\author[Liu]{Shiping Liu}
\address{S. Liu, School of Mathematical Sciences, University of Science and Technology of China, Hefei 230026, China}
\email{spliu@ustc.edu.cn}

\author[Peyerimhoff]{Norbert Peyerimhoff}
\address{N. Peyerimhoff, Department of Mathematical Sciences, Durham University, Durham DH1 3LE, United Kingdom}
\email{norbert.peyerimhoff@durham.ac.uk}

\maketitle

\begin{abstract}
  In this paper, we compare Ollivier Ricci curvature and Bakry-\'Emery
  curvature notions on combinatorial graphs and discuss connections to
  various types of Ricci flatness. We show that non-negativity of
  Ollivier Ricci curvature implies non-negativity of Bakry-\'Emery
  curvature under triangle-freeness and an additional in-degree
  condition. We also provide examples that both conditions of this
  result are necessary. We investigate relations to graph products and
  show that Ricci flatness is preserved under all natural
  products. While non-negativity of both curvatures are preserved
  under Cartesian products, we show that in the case of strong
  products, non-negativity of Ollivier Ricci curvature is only
  preserved for horizontal and vertical edges. We also prove that all
  distance-regular graphs of girth $4$ attain their maximal possible
  curvature values.
\end{abstract}

\section{Introduction}

\subsection{Motivation of the paper}

Curvature is a fundamental notion in the setting of smooth Riemannian
manifolds. There is no unique choice of an analogue of curvature in
the setting of combinatorial graphs. Two possibilities are
\emph{Ollivier Ricci curvature} and \emph{Bakry-\'Emery curvature}
which are both motivated by specific curvature properties of
Riemannian manifolds.  Ollivier Ricci curvature, introduced in \cite{Oll}, is based on the observation that, in the case of positive/negative Ricci curvature, average distances between corresponding point in two nearby small balls in Riemannian manifolds are smaller/larger than the distance between their centres. This fact is reinterpreted using the theory of Optimal Transportation of probability measures representing these balls. 
Bakry-\'Emery curvature, introduced in \cite{BE85}, is based on the so-called \emph{curvature-dimension inequality} which reads for $n$-dimensional Riemannian manifolds $(M,g)$ as follows:
\begin{equation} \label{eq:curv-dim-ineq}
\frac{1}{2} \Delta \Vert {\rm{grad}{f}} \Vert^2(x) \ge \langle \nabla f(x),
\nabla \Delta f(x) \rangle + \frac{1}{n} (\Delta f(x))^2 +
{\rm{Ric}}(\nabla f, \nabla f)(x)
\end{equation}
for all $f \in C^\infty(M)$ and $x \in M$. Here, ${\rm{Ric}}(v,w)$ for
tangent vectors $v,w$ at $x$ stands for the Ricci curvature of the
manifold. This formula is a straightforward implication of Bochner's
identity, a fundamental fact in Riemannian Geometry with many
important consequences. Both curvature notions have been further discussed in the setting of graphs in several literatures (see, e.g., \cite{LLY} for Ollivier Ricci curvature and \cite{Elworthy,LY,Schm99} for Bakry-\'Emery curvature). For the precise definitions of both notions in this paper, we refer to Section \ref{sec:curvnotions}.

While there are many special cases in which these two discrete
curvature notions are related, it is a challenging problem to develop
a satisfactory general understanding of the agreements and differences
of these two curvature notions.

One special family of graphs which have both non-negative Ollivier
Ricci curvature and non-negative Bakry-\'Emery curvature was
introduced by F.R.K. Chung and S.-T. Yau \cite{ChY96}, namely
\emph{Ricci flat graphs}. The notion of Ricci flatness was motivated
by the structure of the $d$-dimensional grid $\Z^d$ (with vanishing
Ollivier Ricci and Bakry-\'Emery curvature) and the class of Ricci
flat graphs contains all abelian Cayley graphs as a subset. 

The motivation of this paper is to investigate various relations
between these two curvature notions and the property of Ricci flatness
with special focus on triangle-free graphs. We also present explicit
examples of graphs related to our results. The curvatures of these
examples were calculated numerically via the interactive
web-application at
\begin{center}
\url{https://www.mas.ncl.ac.uk/graph-curvature/}
\end{center}
For more details about this very useful tool we refer the readers to
\cite{CKLLS2018}.

\subsection{Statement of results}

Let $G=(V,E)$ be a regular graph. Ollivier Ricci curvature
$\kappa_p(x,y)$ is defined on edges $\{x,y\} \in E$ and there is an
idleness parameter $p \in [0,1]$ involved. Lin, Lu, and Yau introduced in \cite{LLY} a modified notion of Ollivier Ricci curvature, denoted by
$\kappa_{LLY}(x,y)$. Both notions are introduced in Definition
\ref{defn:Ollcurv}. While it is known that
$\kappa_0 \le \kappa_{LLY}$, our first result shows in Subsection \ref{sec:curvOR} that positive
$\kappa_{LLY}$-curvature implies non-negativity of
$\kappa_0$-curvature:

\begin{thm} \label{thm:kLLYk0comp} Let $G=(V,E)$ be a regular
  graph. Then we have the following implication for all edges
  $\{x,y\} \in E$:
  $$ \kappa_{LLY}(x,y) > 0 \quad \Longrightarrow \quad \kappa_0(x,y) \ge 0. $$
\end{thm}

The Bakry-\'Emery curvature is defined on vertices and the above
inequality \eqref{eq:curv-dim-ineq} involves a dimension parameter
$n$. Since graphs do not have a well-defined dimension, a natural
choice simplifying this inequality is $n = \infty$. The induced
Bakry-\'Emery curvature value at a vertex $x$ is then denoted by
$\K_\infty(x)$ (see Definition \ref{defn:BEcurvature}).

\medskip

Let us now turn to the above mentioned notion of Ricci flatness. Ricci
flatness is defined locally for individual vertices. In this paper we
also introduce stronger types of Ricci flatness, namely $(R)$-, $(S)$- and $(RS)$-Ricci flatness (see Definition \ref{defn:ricci_flat} below). A fundamental consequence of Ricci flatness is that it implies both
non-negativity of Ollivier Ricci and Bakry-\'Emery curvatures; the
stronger property of $(R)$-Ricci flatness implies even strict
positivity of these two curvatures (see Theorems \ref{thm:ORcurvRF}
and \ref{thm:BEcurvRF}).

Another basic property of Ricci flatness is that it is preserved under
natural graph products (see Theorem \ref{thm:preserv-Rf}). The graph
products under consideration namely, Cartesian product (involving
horizontal and vertical edges), tensorial product (involving only
diagonal edges), and the strong product (involving all three types of
edges), are introduced in Definition \ref{def:graph-products}
below. While Cartesian products preserve non-negativity of both
Ollivier Ricci curvature and Bakry-\'Emery curvature, in the case of
strong products, non-negative Ollivier Ricci curvature is only
preserved for horizontal and vertical edges (see Corollary
\ref{cor:strong_prod_nonneg}).

\medskip

We also consider the case of graphs which contain no triangles. In Section \ref{sec:triangle-free}, we present our main result of this paper relating the two curvature notions. P. Ralli \cite{Ralli} gave an
interesting criterion for curvature sign agreement of both curvature
notions for triangle-free graphs which do not contain the complete
bipartite graph $K_{2,3}$ as a subgraph. He mentions that the
situation is much more unclear if one restricts to general
triangle-free graphs. Our result requires triangle-freeness at a
vertex $x$ and the additional assumption that the in-degrees of
vertices in the $2$-sphere $S_2(x)$ are smaller or equal to $2$. This
assumption is weaker than non-existence of $K_{2,3}$ as a subgraph.

\begin{thm} \label{cor:curvimp}
  Given a regular graph $G=(V,E)$, let $x \in V$ be a vertex not
  contained in a triangle and satisfying $d_x^-(z) \le 2$ for all
  $z \in S_2(x)$. Then we have the following:
  \begin{itemize}
  \item[(a)] $\kappa_0(x,y) = 0$ for all $y \in S_1(x)$ implies
    $\K_\infty(x) \ge 0$.
  \item[(b)] $\kappa_{LLY}(x,y) = \frac{2}{d}$ for all $y \in S_1(x)$ implies $\K_\infty(x) = 2$.
  \end{itemize}
\end{thm}
It is an important remark here that $\kappa_0(x,y) = 0$, $\kappa_{LLY}(x,y) = \frac{2}{d}$, and $\K_\infty(x) = 2$ are the maximum possible values of curvature for a vertex $x$ not contained in a triangle. This curvature comparison result is proved by employing Ricci flatness, see Section \ref{sec:triangle-free}. At the end of the section, we provide also examples to show that all conditions of the
theorem are necessary.

In the final Section \ref{sec:distreg}, we show that the curvatures of
all \emph{distance-regular graphs of girth $4$} and vertex degree $d$
satisfy $\kappa_0 = 0$, $\kappa_{LLY} = \frac{2}{d}$ and
$\K_\infty = 2$ (see Theorem \ref{thm:distreg}). In other words, all
curvatures attain their maximal possible values for this interesting
family of triangle-free graphs.

\section{Curvature notions}
\label{sec:curvnotions}

All graphs $G=(V,E)$ with vertex set $V$ and edge set $E$ in this
paper are simple (that is, without loops and multiple edges),
undirected and connected, and we assume that the vertex degrees $d_x$
of all vertices $x \in V$ are finite. Moreover, all our graphs are
regular (that is $d_x = d$ for all $x \in V$) unless stated otherwise.
Balls and spheres are denoted by
\begin{eqnarray*}
  B_k(x) &:=& \{ z \in V: d(x,z) \le k \}, \\
  S_k(x) &:=& \{ z \in V: d(x,z) = k \},
\end{eqnarray*}
where $d: V \times V \to \N \cup \{ 0 \}$ is the combinatorial
distance function.

\subsection{Ollivier Ricci curvature}
\label{sec:curvOR}

We define the following probability distributions $\mu^p_x$ for any
$x\in V,\: p\in[0,1]$:
$$\mu_x^p(z)=\begin{cases}p,&\text{if $z = x$,}\\
\frac{1-p}{d_x},&\text{if $z\sim x$,}\\
0,& \mbox{otherwise.}\end{cases}$$

\begin{defn}[Transport plan and Wasserstein distance]
Given $G = (V,E)$, let $\mu_{1},\mu_{2}$ be two probability measures on $V$. A \emph{transport plan} $\pi$ transporting $\mu_1$ to $\mu_2$ is a function $\pi:V\times  V\rightarrow [0,\infty)$ satisfying the following marginal constraints
\begin{equation} \label{eq:marginalcond}
\mu_{1}(x)=\sum_{y\in V}\pi(x,y),\:\:\:\mu_{2}(y)=\sum_{x\in V}\pi(x,y).
\end{equation}
The \emph{cost} of a transport plan $\pi$ is given by
\begin{equation*} 
{\rm cost}(\pi)=\sum_{y\in V}\sum_{x\in V} d(x,y)\pi(x,y).
\end{equation*}
The set of all transport plans satisfying \eqref{eq:marginalcond} is
denoted by $\Pi(\mu_{1},\mu_{2})$.

The \emph{Wasserstein distance} $W_1(\mu_{1},\mu_{2})$ between $\mu_{1}$ and $\mu_{2}$ is then defined as
\begin{equation} \label{eq:W1def}
W_1(\mu_{1},\mu_{2}):=\inf_{\pi} {\rm cost}(\pi) = \inf_{\pi} \sum_{y\in V}\sum_{x\in V} d(x,y)\pi(x,y),
\end{equation}
where the infimum runs over all transport plans $\pi\in\Pi(\mu_{1},\mu_{2})$.
\end{defn}

\begin{remark}
  Note that every $\pi \in \Pi(\mu_1,\mu_2)$ satisfies $\pi(x,y) = 0$
  if $x \not\in \supp(\mu_1)$ or $y \not\in \supp(\mu_2)$. Therefore
  \eqref{eq:W1def} can be rewritten as
  $$ W_1(\mu_{1},\mu_{2})=\inf_{\pi}
  \sum_{y\in\supp(\mu_2)}\sum_{x\in \supp(\mu_1)} d(x,y)\pi(x,y).
  $$
\end{remark}

In other words, a transport plan $\pi$ moves a mass
distribution given by $\mu_1$ into a mass distribution given by
$\mu_2$, and $W_1(\mu_1,\mu_2)$ is a measure for the minimal effort
which is required for such a transition.

If $\mu_1$ and $\mu_2$ have finite supports, then there exists $\pi$ which attains the infimum in \eqref{eq:W1def}. We call such $\pi$ an {\it optimal transport plan} transporting $\mu_{1}$ to $\mu_{2}$.

\begin{defn}[Ollivier Ricci curvature] \label{defn:Ollcurv} The
  \emph{$p$-Ollivier Ricci curvature} \cite{Oll} on an edge
  $\{x,y\} \in E$ is
$$\kappa_{ p}(x,y)=1-W_1(\mu^{ p}_x,\mu^{ p}_y),$$
where $p\in [0,1]$ is called the \emph{idleness} parameter.

The Ollivier Ricci curvature introduced by Lin, Lu, and Yau \cite{LLY}, is
defined as
$$\kappa_{LLY}(x,y) = \lim_{ p\rightarrow 1}\frac{\kappa_{ p}(x,y)}{1- p}.$$
\end{defn}

It was shown in \cite[Lemma 2.1]{LLY} that the function $p \mapsto \kappa_p(x,y)$ is concave, which implies
\begin{equation} \label{eq:k0lessklly}
\kappa_p(x,y) \le \kappa_{LLY}(x,y) \quad \text{for all $p \in [0,1]$.}
\end{equation}
Moreover, we have the following relation for edges $\{x,y\}$ with
$d_x=d_y=d$ (see \cite{I}):
\begin{equation} \label{eq:LLYOllconn}
\kappa_{LLY}(x,y) = \frac{d+1}{d}\kappa_{\frac{1}{d+1}}(x,y).
\end{equation}

From the definition of the Wasserstein metric we can get an upper bound for $W_1$ by choosing a suitable transport plan. Using Kantorovich duality (see e.g. \cite[Ch. 5]{Vill09}), a fundamental concept in the optimal transport theory, we can approximate the opposite direction:

\begin{thm}[Kantorovich duality]\label{Kantorovich}
Given $G=(V,E)$, let $\mu_{1},\mu_{2}$ be two probability measures on $V$. Then
$$W_1(\mu_{1},\mu_{2})=\sup_{\substack{\phi:V\rightarrow \mathbb{R}\\ \phi\in {\rm \emph{$1$-Lip}}}}  \sum_{x\in V}\phi(x)(\mu_{1}(x)-\mu_{2}(x)),$$
where \emph{$1$-Lip} denotes the set of all $1$-Lipschitz functions. 
If $\phi \in$ \emph{$1$-Lip} attains the supremum we call it an \emph{optimal Kantorovich potential} transporting $\mu_{1}$ to $\mu_{2}$.
\end{thm}

Note that both curvatures $\kappa_0(x,y)$ and $\kappa_{LLY}(x,y)$ of
an edge $\{x,y\}$ are already determined by the combinatorial
structure of the induced subgraph $B_2(x)$. (In fact, by symmetry
reasons, the combinatorial structure of the induced subgraph $B_2(x)
\cap B_2(y)$ is sufficient.)

As the relation $\kappa_0 \le \kappa_{LLY}$ is known from \eqref{eq:k0lessklly}, now we will prove the surprising fact that strict
positivity of $\kappa_{LLY}$ implies non-negativity of $\kappa_0$ (as stated in Theorem \ref{thm:kLLYk0comp} from the Introduction).

\begin{proof}[Proof of Theorem \ref{thm:kLLYk0comp}]
  Let $G=(V,E)$ be $d$-regular. Using the relation
  \eqref{eq:LLYOllconn}, it suffices to prove
  $$ \kappa_{\frac{1}{d+1}}(x,y) > 0 \quad \Longrightarrow \quad
  \kappa_0(x,y) \ge 0. $$
  Let $\{x,y\} \in E$ be an edge with
  $\kappa_{\frac{1}{d+1}}(x,y) > 0$. We define the following sets:
  \begin{eqnarray*}
    T_{xy} &:=& S_1(x) \cap S_1(y), \\
    V_x&:=& S_1(x) \backslash B_1(y), \\
    V_y &:=& S_1(y) \backslash B_1(x).
  \end{eqnarray*}
  In other words, $T_{xy}$ is the set of common neighbours of $x$ and $y$,
  $V_x$ is the set of neighbours of $x$ which have distance $2$ to $y$
  and, similarly, $V_y$ is the set of neighbours of $y$ which have distance $2$
  to $x$. 

We can choose an optimal transport plan $\pi_{opt} \in \Pi(\mu_x^{\frac{1}{d+1}},\mu_y^{\frac{1}{d+1}})$ with 
\begin{enumerate}
\item[i)] if $u\in T_{xy}\cup\{x\}\cup \{y\}$, then $\pi_{opt}(u,u)= \frac{1}{d+1}$, 
\item[ii)] if $u\in V_x$, then $\pi_{opt}(u,v)=\frac{1}{d+1}$ for exactly one $v\in V_y$ and 0 for others, 
\item[iii)] if $u\notin B_1(x)$, then  $\pi_{opt}(u,v)=0$ for $v \in V$. 
\end{enumerate}
The existence of an optimal transport plan satisfying (ii) (that is,
without splitting mass), follows from \cite[Theorem 1.1]{Brezis} (see also
\cite[p. 5]{Villani}). Moreover, this transport plan can be chosen to
satisfy (i) by \cite[Lemma 4.1]{I}. Note that (iii) holds for any transport plan in $\Pi(\mu_x^{\frac{1}{d+1}},\mu_y^{\frac{1}{d+1}})$.

In other words, the optimal transport plan does not move the mass
distributions at $x$, $y$ or $T_{xy}$, and for the vertices in $V_x$
it moves the mass distribution from one vertex completely to one
vertex in $V_y$.  Thus the optimal transport plan pairs the vertices
at $V_x$ and $V_y$. Let $u\in V_x$ and denote by $\tilde{u}$ the
unique vertex in $V_y$ for which
$\pi_{opt}(u,\tilde{u})=\frac{1}{d+1}$.

Let us then consider the Wasserstein distance. Using the optimal
transport plan we can write
\begin{equation} \label{eq:onemorethanWasserstein}
1>1-\kappa_\frac{1}{d+1}(x,y)= W_1(\mu_x^\frac{1}{d+1}, \mu_y^\frac{1}{d+1}) = \frac{1}{d+1}\sum_{u \in V_x} d(u,\tilde{u}).
\end{equation}
Note that $1\leq d(u_j,\tilde{u}_j)\leq 3$ for all $u_j\in V_x$. Let
$$ N_i := | \{ u \in V_x: d(u,\tilde u) = i \} | \quad \text{for $i \in \{1,2,3\}$.} $$
It follows from \eqref{eq:onemorethanWasserstein} that 
$d+1> \sum_{u \in V_x} d(u,\tilde{u})=N_1+2N_2+3N_3$,
which implies
\begin{equation} \label{eq:importantdistanceformula}
d \ge N_1+2N_2+3N_3.
\end{equation}
Now we distinguish three cases.

\medskip

Assume that $N_3 > 0$. Then there exists at least one vertex
$w \in V_x$ satisfying $d(w,\tilde{w})=3$. Let $\pi$ be a transport
plan from $\mu_x^0$ to $\mu_y^0$ such that $\pi(w,x)=\frac{1}{d}$,
$\pi(y,\tilde{w})=\frac{1}{d}$ and $\pi(u,\tilde{u})=\frac{1}{d}$ for
all other pairs $(u,\tilde{u})$ on the support of $\pi_{opt}$ except
$(w,\tilde{w})$. Using this transport plan and
\eqref{eq:importantdistanceformula}, we have
\begin{eqnarray*}
W_1(\mu_x^0,\mu_y^0)&\leq& \frac{1}{d}\left(2+N_1+2N_2+3(N_3-1)\right)\\
&\leq& \frac{d-1}{d} <1.
\end{eqnarray*}
Thus $\kappa_0(x,y)>0$.

\medskip

Next, we assume $N_3 = 0$ and $N_2 > 0$. Then there exists at least one vertex  $w \in V_x$ satisfying $d(w,\tilde{w})=2$, and we obtain, similarly as above,
\begin{eqnarray*}
W_1(\mu_x^0,\mu_y^0) &\leq& \frac{1}{d}(2+N_1+2(N_2-1))\\
                    &\leq& \frac{1}{d}(N_1+2N_2+3N_3) \leq 1,
\end{eqnarray*}
and therefore $\kappa_0(x,y)\geq 0$.

\medskip

Finally, if $N_2=N_3=0$, the optimal transport plan $\pi_{opt}$ defines a perfect matching between the sets $V_x$ and $V_y$, and therefore 
$$ W_1(\mu_x^0,\mu_y^0)\leq \frac{2+(N_1-1)}{d} = \frac{N_1+1}{d} \le 1,$$
since $N_1 = |V_x| \le d-1$, and again, $\kappa_0(x,y) \ge 0$, with equality if and only if $N_1=d-1$, which means $T_{xy} = \emptyset$.
\end{proof}

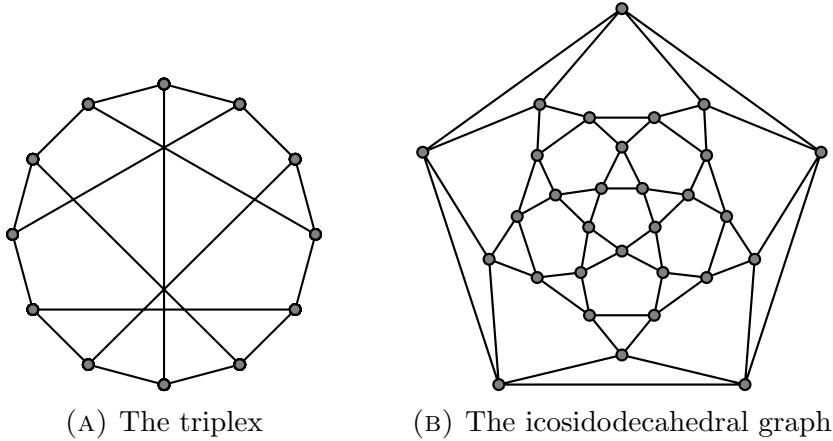
\begin{figure}[h!]
\begin{center}
\tikzstyle{every node}=[circle, draw, fill=black!50,
                        inner sep=0pt, minimum width=4pt]

    \begin{subfigure}[b]{0.35\textwidth}
        \centering
        \resizebox{\linewidth}{!}{
 \begin{tikzpicture}[thick,scale=1]%
 \draw \foreach \x in {0,30,...,330} {
        (\x:2)  -- (\x+30:2)    
		(90:2) node(x1){}
		(60:2) node(x2){}
		(30:2) node(x3){}
		(0:2) node(x4){}
		(-30:2) node(x5){}
		(-60:2) node(x6){}
		(120:2) node(x12){}
		(150:2) node(x11){}
		(180:2) node(x10){}
		(210:2) node(x9){}
		(240:2) node(x8){}
		(270:2) node(x7){}		
		};
 \draw (x1)--(x7) (x2)-- (x10) (x3)--(x8) (x4)--(x12) (x5)-- (x9) (x6)-- (x11);		
\end{tikzpicture}    
}                        
\caption{The triplex}
\label{fig:triplex}
    \end{subfigure}
\qquad
    \begin{subfigure}[b]{0.45\textwidth}
        \centering
        \resizebox{\linewidth}{!}{
 \begin{tikzpicture}[thick,scale=.45]%
    \draw \foreach \x in {18,90,...,306} {
        (\x:6) node{} -- (\x+72:6)
        (\x:6) -- (\x+36:4) node{}
		(\x:6) -- (\x-36:4)
		(\x+36:4) -- (\x+18:3)
		(\x+36:4) -- (\x+54:3)
		(\x+18:3) node{} -- (\x-18:3) node{}
		(\x+54:3)  -- (\x+18:3) 
		(\x:2) -- (\x-18:3)
		(\x:2) -- (\x+18:3)
        (\x:2) node{} -- (\x+36:1) node{}
		(\x:2) -- (\x-36:1)
		(\x+36:1) -- (\x-36:1)
};
\end{tikzpicture}  
}                        
\caption{The icosidodecahedral graph}
\label{fig:icosidodecahedron}
    \end{subfigure}
   \end{center}
   \caption{Examples of graphs with $\kappa_{LLY}=0$}
   \label{fig:OR_examples2}
\end{figure}

\begin{remark} \label{rem:ollcurvex}
  \ \\ 
  \noindent
  (a) The proof shows that $\kappa_{LLY}(x,y) > 0$ implies
  $\kappa_0(x,y) > 0$ in the following cases:
  \begin{enumerate}
  \item[(i)] $N_3 > 0$ or
  \item[(ii)] $N_3 = N_2 = 0$ and $\{x,y\}$ is contained in a triangle.
  \end{enumerate}

  \noindent
  (b) The hypercubes $Q^d$ satisfy $\kappa_{LLY}(x,y) = \frac{2}{d} > 0$
  and $\kappa_0(x,y) = 0$ for all edges $\{x,y\} \in E$.

  \smallskip
  
  \noindent
  (c) The triplex (see Figure \ref{fig:triplex}) satisfies
    $\kappa_{LLY}(x,y) = 0$ and $\kappa_0(x,y) = - \frac{1}{3} < 0$
    for all edges $\{x,y\} \in E$.

  \smallskip  
    
  \noindent  
  (d) The icosidodecahedral graph (see Figure
  \ref{fig:icosidodecahedron}) satisfies $\kappa_{LLY}(x,y) = 0$ and
  $\kappa_0(x,y) = 0$ for all edges $\{x,y\} \in E$. This implies that
  $\kappa_p(x,y) = 0$ for all $p \in [0,1]$. Graphs with this property
  in all edges are called \emph{bone-idle} (this notion was introduced
  in \cite{I}).

  \medskip

  \noindent
  The examples (b) and (c) show that the result in the theorem is sharp.
\end{remark}

We finish this subsection with the following upper curvature bounds for $\kappa_0$ and $\kappa_{LLY}$:

\begin{thm}[see {\cite[Theorem 4]{JL} and \cite[Proposition 2.7]{rigidity}}] \label{thm:uppbdLLY}
  Let $G = (V,E)$ be $d$-regular and $\{x,y\} \in E$. Then
  $$ \kappa_0(x,y) \le \frac{\#_\Delta(x,y)}{d}, $$
  and
  $$ \kappa_{LLY}(x,y) \le \frac{2+\#_\Delta(x,y)}{d}, $$
  where $\#_\Delta(x,y)$ is the number of triangles containing $\{x,y\}$.
\end{thm}

\subsection{Bakry-\'Emery curvature}
\label{sec:curvBE}

This curvature notion was first introduced by Bakry and \'Emery in
\cite{BE85} and was applied on graphs in \cite{Elworthy, LY, Schm99}. The definition of this curvature is based on the
curvature-dimension inequality \eqref{eq:curv-dim-ineq}, which is
equivalently rewritten as \eqref{eq:CDineq} below with the help of the
following $\Gamma$-calculus.

For any function $f: V\to \mathbb{R}$ and any
vertex $x\in V$, the (non-normalized) \emph{Laplacian} $\Delta$ is
defined via
\begin{equation*}
\Delta f(x):=\sum_{y:y\sim x}(f(y)-f(x)).
\end{equation*}

\begin{defn}[$\Gamma$ and $\Gamma_{2}$ operators]\label{defn:GammaGamma2}
Given $G=(V,E)$, we define for two functions $f,g: V\to \mathbb{R}$
\begin{align*}
2\Gamma(f,g)&:=\D(fg)-f\D g-g\D f;\\
2\Gamma_2(f,g)&:=\D\Gamma(f,g)-\Gamma(f,\D g)-\Gamma(\D f,g).
\end{align*}
We write $\Gamma(f):=\Gamma(f,f)$ and $\Gamma_2(f,f):=\Gamma_2(f)$, for short.
\end{defn}

\begin{defn}[Bakry-\'Emery curvature]\label{defn:BEcurvature} Given
  $G=(V,E)$, $\K\in \mathbb{R}$ and $\NN \in (0,\infty]$. We say that a
  vertex $x\in V$ satisfies the \emph{curvature-dimension inequality}
  $CD(\mathcal{K},\mathcal{N})$, if for any $f:V\to \mathbb{R}$, we
  have
\begin{equation}\label{eq:CDineq}
  \Gamma_2(f)(x)\geq \frac{1}{\NN}(\Delta f(x))^2+\K\Gamma(f)(x) \quad
  \text{for all $x \in V$.}
\end{equation}
We call $\K$ a lower Ricci curvature bound of $x$, and $\NN$ a dimension parameter. The graph $G=(V,E)$ satisfies $CD(\K,\NN)$ (globally), if all its vertices satisfy $CD(\K,\NN)$. At a vertex $x\in V$, let $\K(x,\NN)$ be the largest $\K$ such that (\ref{eq:CDineq}) holds for all functions $f$ at $x$ for a given $\NN$. We call $\K(x,\cdot)$ the \emph{Bakry-\'Emery curvature function} of $x$ and
we define
$$ \K_\infty(x) := \lim_{\NN \to \infty} \K(x,\NN). $$ 
\end{defn}


In this paper, we will restrict our considerations to the
curvature at $\infty$-dimension $\K_\infty: V \to \R$. Note that for the definition of 
$\K_\infty(x)$, the formula \eqref{eq:CDineq} simplifies to
$$ \Gamma_2(f)(x)\geq \K\Gamma(f)(x) \quad
\text{for all $x \in V$.} $$
The quadratic forms $\Gamma(\cdot,\cdot)(x)$ and $\Gamma_2(\cdot,\cdot)(x)$ can
be represented by matrices $\Gamma(x)$ and $\Gamma_2(x)$ as follows
\begin{eqnarray*}
  \Gamma(f,g)(x) &=& \underline{f} \Gamma(x) \underline{g}^\top, \\
  \Gamma_2(f,g)(x) &=& \underline{f} \Gamma_2(x) \underline{g}^\top,
\end{eqnarray*}
where $\underline{f},\underline{g}$ are the vector representations of
$f$ and $g$. The matrices $\Gamma(x), \Gamma_2(x)$ are symmetric with
non-zero entries only in $B_1(x)$ and $B_2(x)$, respectively. So we
can view them as local matrices by disregarding the vertices outside
$B_2(x)$. For the explicit matrix entries of $\Gamma(x)$ and
$\Gamma_2(x)$ see \cite[Subsections 2.2 and 2.3]{CLP}. Note that these
entries are already fully determined by the combinatorial structure of
the \emph{incomplete $2$-ball around $x$}, denoted by
$B^{\rm{inc}}_2(x)$, which is the induced subgraph of $B_2(x)$ with
all edges within $S_2(x)$ removed.

We have the following general upper curvature bound similar to Theorem 
\ref{thm:uppbdLLY}:

\begin{thm}[see {\cite[Corollary 3.3]{CLP}}] \label{thm:uppbdBE}
  Let $G=(V,E)$ be $d$-regular and $x \in V$. Then
  $$ \K_\infty(x) \le 2 + \frac{\#_\Delta(x)}{d}, $$
  where $\#_\Delta(x)$ is the number of triangles containing $x$.
\end{thm}

Let us finally return to the examples from the previous subsection.

\begin{remark} \label{rem:excurv}
  The examples in Remark \ref{rem:ollcurvex} have the following
  Bakry-\'Emery and Ollivier Ricci curvatures:

  \medskip
  
  \begin{tabular}{l|ccc}
    & $\kappa_0(x,y)$ & $\kappa_{LLY}(x,y)$ & $\K_\infty(x)$ \\
    \hline\\[-.2cm]
    Hypercube $Q^d$ & $0$ & $\frac{2}{d}$ & $2$ \\[.2cm]
    Triplex & $-\frac{1}{3}$ & $0$ & $-1$ \\[.2cm]
    Icosidodecahedral graph & $0$ & $0$ & $-\frac{3}{2}$
  \end{tabular}

  \medskip
  
  None of the regular graphs in the above table have curvature with
  opposite signs. We are not aware of any such examples and it would be
  interesting to find such graphs.
\end{remark}


\section{Ricci flat graphs}

The notion of Ricci flat graphs was introduced in 1996 by Chung and
Yau \cite{ChY96} in connection to a logarithmic Harnack inequality and
is motivated by the structure of the $d$-dimensional grid $\Z^d$. Abelian Cayley graphs are prominent examples of Ricci flat graphs.


\begin{defn} \label{defn:ricci_flat}
Let $G=(V,E)$ be a $d$-regular graph. We say that $x \in V$ is \emph{Ricci flat} if there exist maps
$\eta_i: B_1(x) \to V$ for $1 \le i \le d$ with the following properties:
\begin{itemize}
\item[(i)] $\eta_i(u) \sim u$ for all $u\in B_1(x)$,
\item[(ii)] $\eta_i(u) \neq \eta_j(u)$ if $i \neq j$,
\item[(iii)] $\bigcup_j \eta_j(\eta_i x)) = \bigcup_j \eta_i(\eta_j x)$ for all $i$.  
\end{itemize}
We also consider the following additional properties of the maps $\eta_i$:
\begin{itemize}
\item[(R)] Reflexivity: $\eta_i^2(x) = x$ for all $i$,
\item[(S)] Symmetry: $\eta_j(\eta_i x) = \eta_i(\eta_j x)$ for all $i,j$.
\end{itemize}
If there exists a family of maps $\eta_i$ for a given vertex $x \in V$
satisfying property (R) or property (S) in addition to (i)-(iii), we
say that $x$ is $(R)$-Ricci flat or $(S)$-Ricci flat, respectively. If
there exists a family of maps $\eta_i$ satisfying (i)-(iii) and (R)
and (S) simultaneously, we say that $x$ is $(RS)$-Ricci flat.
\end{defn}

The $d$-dimensional grid $\Z^d$ is Ricci flat with the choices
$\eta_i(x) = x + e_i$. The following lemma is a useful observation for the study of Ricci flatness of concrete examples.

\begin{lemma} \label{lem:crucial}
  Assume a family of maps $\eta_i: B_1(x) \to V$ satisfies (i)-(iii)
  of the above definition. Then each of these maps $\eta_i$ is a bijective map
  between $B_1(x)$ and $B_1(\eta_i x)$.
\end{lemma}

\begin{proof}
  Assume that the family $\eta_i$ satisfies (i)-(iii). It follows
  immediately from (i) and (ii) and regularity that
  $$ \bigcup_j \eta_j(u) = S_1(u) \quad \text{for all $u \in B_1(x)$.} $$
  This implies that (iii) is equivalent to
  $$
  S_1(\eta_i x) = \eta_i(S_1(x)) \quad \text{for all $i$,}
  $$
  which, in turn, implies
  \begin{equation} \label{eq:rangeetai}
  B_1(\eta_i x) = S_1(\eta_i x) \cup \{ \eta_i x\} = \eta_i(S_1(x)) \cup
\eta_i(\{x\}) = \eta_i(B_1(x)).
  \end{equation}
  Therefore, each map $\eta_i$ must be injective, since
  $$ |\eta_i(B_1(x))| = |B_1(\eta_i x)| = |B_1(x)|. $$
  Bijectivity from $B_1(x)$ to $B_1(\eta x)$ follows immediately from \eqref{eq:rangeetai}.
\end{proof}

Note that all Ricci flatness properties at a vertex $x$ can be
determined from the combinatorial structure of the incomplete $2$-ball
$B^{\rm{inc}}_2(x)$ around $x$, which was introduced in Subsection
\ref{sec:curvBE}.

\begin{example} To help readers familiarize with the notion of Ricci flatness, we provide three examples of graphs and check whether each of them is Ricci flat.
\begin{itemize}
  \item[(a)] The incomplete $2$-ball in Figure \ref{fig:not_RF} with
  $S_1(x) = \{v_1,v_2,v_3\}$, $v_1 \sim v_2$ and
  $S_2(x) = \{v_4,v_5,v_6\}$, $v_4 \sim v_1$, $v_5 \sim v_2,v_3$ and
  $v_6 \sim v_3$ is not Ricci flat:

\begin{figure}[h!]
\begin{center}
  \tikzstyle{every node}=[circle, draw, fill=black!50,inner sep=0pt, minimum width=4pt]
  \begin{tikzpicture}[thick,scale=1]%
        \draw (0,0) node(x0)[label=above:$x$] {};
        \draw (-2,-1.5) node(y1)[label=left:$v_1$] {};
        \draw (0,-1.5) node(y2)[label=right:$v_2$] {};
        \draw (2,-1.5) node(y3)[label=right:$v_3$] {};
        \draw (-3,-3) node(z4)[label=left:$v_4$] {};
        \draw (1,-3) node(z5)[label=left:$v_5$] {};
        \draw (3,-3) node(z6)[label=right:$v_6$] {};
        \draw (x0)--(y1)--(y2)--(x0);
        \draw (x0)--(y3)--(z6);
        \draw (y1)--(z4);
        \draw (y2)--(z5)--(y3);
\end{tikzpicture}  
\end{center}
\caption{Graph that is not Ricci flat}
\label{fig:not_RF}
\end{figure}
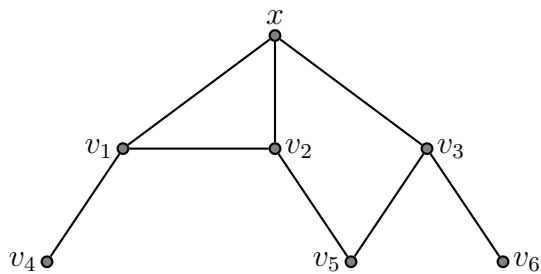

We show this by contradiction. Assume $\eta_i: B_1(x) \to V$ with
properties (i)-(iii) exist. Without loss of generality, we can assume
$\eta_i(x) = v_i$.  Note that we must have
$\eta_i(v_j) \in S_1(v_i) \cap S_1(v_j)$ for $1 \le i,j \le d$. This
implies that we have the following choices for our maps $\eta_j$:
  \begin{center}
  \begin{tabular}{c|c|c|c|c}
    & $x$ & $v_1$ & $v_2$ & $v_3$ \\
    \hline
    $\eta_1$ & $v_1$ & $x,v_2,v_4$ & $x$ & $x$ \\
    $\eta_2$ & $v_2$ & $x$ & $x,v_1,v_5$ & $x,v_5$ \\
    $\eta_3$ & $v_3$ & $x$ & $x,v_5$ & $x,v_5,v_6$
   \end{tabular}
   \end{center}
   Such a table can be presented concisely with the
   help of a $d \times d$ matrix $A$, namely, $A = (A_{ij})$ defined as follows:
   Let $S_1(x) = \{v_1,\dots, v_d\}$ where $v_j := \eta_j(x)$, and
   $S_2(x) =: \{v_{d+1},\dots,v_t\}$ and, furthermore, $v_0 := x$.  Then
   the entries $A_{ij} \in \{0,1,\dots,t\}$ of $A$ are given via the
   relation
   $$ v_{A_{ij}} = \eta_i(v_j). $$
   Then the table translates into the following possibilities for the entries
   of $A$:
   $$ \begin{pmatrix} 0,2,4 & 0 & 0 \\
     0 & 0,1,5 & 0,5 \\
     0 & 0,5 & 0,5,6 \end{pmatrix}. $$
   The conditions (i)-(iii) require that all columns and rows of $A$
   have non repeating entries. Obviously, this is not possible in this
   case. Henceforth, we will use this matrix notation to simplify matters.
   \medskip
   \item[(b)] The graph $K_{3,3}$: Let $S_1(x) = \{v_1,v_2,v_3\}$ and
   $S_2(x) = \{v_4,v_5\}$ with $v_4,v_5 \sim v_1,v_2,v_3$. We have the
   following possibilities for the entries of the associated matrix $A$:
   $$ \begin{pmatrix} 0,4,5 & 0,4,5 & 0,4,5 \\
     0,4,5 & 0,4,5 & 0,4,5 \\
     0,4,5 & 0,4,5 & 0,4,5
   \end{pmatrix}.
   $$
   Note that $(R)$-Ricci flatness requires existence of an associated
   matrix $A$ with vanishing diagonal and $(S)$-Ricci flatness
   requires existence of a symmetric matrix $A$. Therefore, $x$ is
   $(R)$- and $(S)$-Ricci flat by the following matrix choices:
   $$ A_R = \begin{pmatrix} 0 & 4 & 5 \\
     5 & 0 & 4 \\
     4 & 5 & 0 \end{pmatrix}, \qquad A_S = \begin{pmatrix} 0 & 4 & 5 \\
     4 & 5 & 0 \\
     5 & 0 & 4 \end{pmatrix}.
   $$
   Note that $x$ is not $(RS)$-Ricci flat since both properties
   (vanishing diagonal and symmetry) cannot be satisfied at the same
   time. In fact, the complete bipartite graphs $K_{d,d}$ are both
   $(R)$- and $(S)$-Ricci flat for all $d$, and $(RS)$-Ricci flat if
   and only if $d$ is even (see the Appendix).
   \medskip
   \item[(c)] Shrikhande graph: Cayley graph $\Z_4 \times Z_4$ with the generator set $\{\pm(0,1), \pm(1,0), \pm(1,1)\}$. It is a strongly
   regular graph (see \cite[pp. 125]{BH}). The structure of the
   incomplete $2$-ball $B_2^{\rm{inc}}(x)$ around any vertex $x$ is
   given in Figure \ref{fig:Shrikhande}. We have the following
   possibilities for the entries of the associated matrix $A$:
   $$ \scriptsize{\begin{pmatrix} 0,2,6,7,12,15 & 0,7 & 0,2 & 0,12 & 0,6 & 0,15 \\
       0,7 & 0,1,3,7,8,13 & 0,8 & 0,3 & 0,13 & 0,1 \\
       0,2 & 0,8 & 0,2,4,8,9,14 & 0,9 & 0,4 & 0,14 \\
       0,12 & 0,3 & 0,9 & 0,3,5,9,10,12 & 0,10 & 0,5 \\
       0,6 & 0,13 & 0,4 & 0,10 & 0,4,6,10,11,13 & 0,11 \\
       0,15 & 0,1 & 0,14 & 0,5 & 0,11 & 0,1,5,11,14,15
   \end{pmatrix}}.
 $$
 Choosing $0$ for diagonal entries fixes all other entries of the matrix.
 Moreover, this choice leads to a symmetric matrix, which shows that
 $x$ is $(RS)$-Ricci flat.

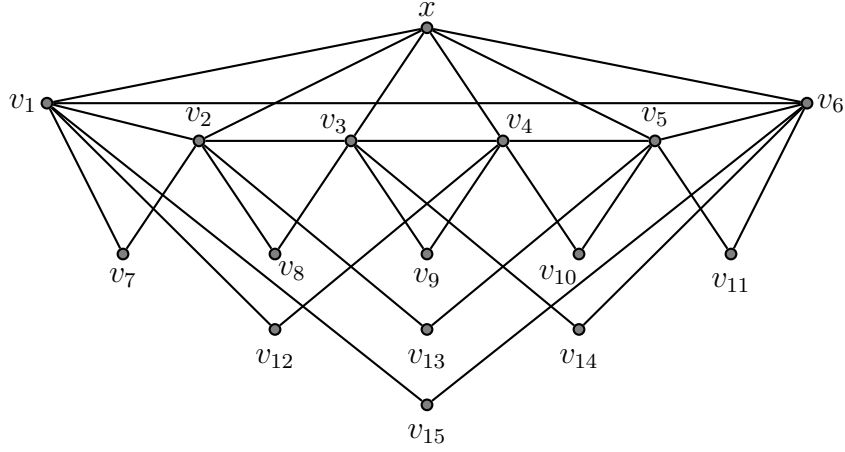
\begin{figure}[h!]
\begin{center}
\tikzstyle{every node}=[circle, draw, fill=black!50,
                        inner sep=0pt, minimum width=4pt]
 \begin{tikzpicture}[thick,scale=1]%
        \draw (0,0) node(x0)[label=above:$x$] {};
        \draw (-5,-1) node(v1)[label=left:$v_1$] {};
        \draw (-3,-1.5) node(v2)[label=above:$v_2$] {};
        \draw (-1,-1.5) node(v3)[label=above left:$v_3$] {};
        \draw (1,-1.5) node(v4)[label=above right:$v_4$] {};
        \draw (3,-1.5) node(v5)[label=above:$v_5$] {};
        \draw (5,-1) node(v6)[label=right:$v_6$] {};
        \draw (v1)--(v2)--(v3)--(v4)--(v5)--(v6)--(v1);
        \draw (x0)--(v1) (x0)--(v2) (x0)--(v3) (x0)--(v4) (x0)--(v5) (x0)--(v6);
        \draw (-4,-3) node(v7)[label=below:$v_7$] {};
        \draw (-2,-3) node(v8)[label=below right:$v_8$] {};
        \draw (0,-3) node(v9)[label=below:$v_9$] {};
        \draw (2,-3) node(v10)[label=below left:$v_{10}$] {};
        \draw (4,-3) node(v11)[label=below:$v_{11}$] {};
        \draw (v1)--(v7)--(v2);
        \draw (v2)--(v8)--(v3);
        \draw (v3)--(v9)--(v4);
        \draw (v4)--(v10)--(v5);
        \draw (v5)--(v11)--(v6);
        \draw (-2,-4) node(v12)[label=below:$v_{12}$] {};
        \draw (0,-4) node(v13)[label=below:$v_{13}$] {};
        \draw (2,-4) node(v14)[label=below:$v_{14}$] {};
        \draw (v1)--(v12)--(v4);
        \draw (v2)--(v13)--(v5);
        \draw (v3)--(v14)--(v6);
        \draw (0,-5) node(v15)[label=below:$v_{15}$] {};
        \draw (v1)--(v15)--(v6);
\end{tikzpicture}  
\end{center}
\caption{The incomplete $2$-ball $B_2^{\rm{inc}}(x)$ of the Shrikhande
  graph}
\label{fig:Shrikhande}
\end{figure}
\end{itemize}  
\end{example}

\subsection{Ricci flatness and Ollivier Ricci curvature}

With regards to Ollivier Ricci curvature we have the following general
implications:

\begin{thm} \label{thm:ORcurvRF}
Let $G = (V,E)$ be $d$-regular.
\begin{itemize}
\item[(a)] If $x \in V$ is Ricci flat then $\kappa_0(x,y) \ge 0$ for all
  edges $\{x,y\} \in E$. 
\item[(b)] If $x \in V$ is $(R)$-Ricci flat then
  $\kappa_{LLY}(x,y) \ge \frac{2}{d}$ for all edges $\{x,y\} \in E$.
\end{itemize}
\end{thm}

\begin{proof}
  For the proof of (a) we assume Ricci flatness at $x$ with
  corresponding maps $\eta_i: B_1(x) \to V$. Let $y \in S_1(x)$. Recall that
  $$ S_1(x) = \{ \eta_1(x),\dots,\eta_d(x) \}. $$
  Therefore, we have $y = \eta_i(x)$ for some $i \in
  \{1,\dots,d\}$. We choose the following transport plan:
  $$ \pi(u,\eta_i(u)) = \frac{1}{d} \quad \text{for all $u \in S_1(x)$,} $$
  and $\pi(u,v) = 0$ for all other combinations. This implies
  $$ \sum_{v \in V} \pi(u,v) = \pi(u,\eta_i(u)) = \frac{1}{d} = \mu_x^0(u) \quad \text{for all $u \in S_1(x)$, }$$
  and (using Lemma \ref{lem:crucial}) 
  $$ \sum_{u \in V} \pi(u,v) = \pi(\eta_i^{-1}(v),v) = \frac{1}{d} = \mu_y^0(v) \quad \text{for all $v \in S_1(y)$,} $$
  which shows that $\pi \in \Pi(\mu_x^0,\mu_y^0)$. This leads to
  $$ W_1(\mu_x^0,\mu_y^0) \le {\rm{cost}}(\pi) = \sum_{u \in S_1(x)} \pi(u,\eta_i(u)) = 1, $$
  which implies $\kappa_0(x,y) \ge 0$.

  We prove (b) similarly. Assume $x$ is $(R)$-Ricci flat with
  corresponding maps $\eta_i$ and $y = \eta_i(x)$. Note that we have
  $\eta_i(y) = x$ from reflexivity. This time, we
  choose the following transport plan
  $\pi \in \Pi(\mu_x^{1/(d+1)},\mu_y^{1/(d+1)})$:
  $$ \pi(u,\eta_i(u)) = \frac{1}{d+1} \quad \text{for all $u \in S_1(x) \backslash \{y\}$,} $$
  $\pi(x,x) = \pi(y,y) = \frac{1}{d+1}$, and $\pi(u,v) = 0$ for all
  other combinations. This leads to
  $$ W_1(\mu_x^{1/(d+1)},\mu_y^{1/(d+1)}) \le {\rm{cost}}(\pi) = \sum_{u \in S_1(x) \backslash \{y\}} \pi(u,\eta_i(u)) = \frac{d-1}{d+1}, $$
  which implies $\kappa_{1/(d+1)}(x,y) \ge \frac{2}{d+1}$ and
  $$ \kappa_{LLY}(x,y) = \frac{d+1}{d} \kappa_{1/(d+1)}(x,y) \ge \frac{2}{d}. $$
\end{proof}

\subsection{Ricci flatness and Bakry-\'Emery curvature}

With regards to Bakry-\'Emery curvature we have the following general
implications:

\begin{thm} \label{thm:BEcurvRF}
Let $G = (V,E)$ be $d$-regular.
\begin{itemize}
\item[(a)] If $x \in V$ is Ricci flat then $\K_\infty(x) \ge 0$. 
\item[(b)] If $x \in V$ is $(R)$-Ricci flat then $\K_\infty(x) \ge 2$.
\end{itemize}
\end{thm}

\begin{proof}
The proof of statement (a) was already explained in \cite{ChY96} and \cite{LY}. This proof stategy can also be applied to prove statement (b). We present these proofs for the reader's convenience.

Recall from the definition that
\begin{equation} \label{eq:Gamma2} 
2 \Gamma_2(f,f)(x) = \Delta \Gamma(f,f)(x) - 2\Gamma(f,\Delta f)(x). 
\end{equation}
and
\begin{equation*}
2 \Gamma(f,g)(x) = \Delta(fg)(x) - f(x)\Delta g(x)- g(x)\Delta f(x). 
\end{equation*}
A useful identity to compute $\Gamma(f,g)$ is
\begin{equation*}
2 \Gamma(f,g)(x) =\sum_{y:y\sim x} (f(y)-f(x))(g(y)-g(x)). 
\end{equation*}

Let us now consider the first term on the RHS in \eqref{eq:Gamma2} and use the identity $A^2 - B^2 = (A-B)^2 + 2B(A-B)$:
\begin{eqnarray*}
\Delta \Gamma(f,f)(x) &=& \sum_{i=1}^d \left( \Gamma(f,f)(\eta_ix) - \Gamma(f,f)(x) \right) \\
&=& \frac{1}{2} \sum_{i=1}^d \left[ \sum_{j=1}^d \left( f(\eta_j \eta_i x) - f(\eta_i x) \right)^2 - \sum_{j=1}^d  \left( f(\eta_jx) - f(x) \right)^2  \right] \\
&=& \sum_{i=1}^d \sum_{j=1}^d \left( f(\eta_j \eta_i x) - f(\eta_i x) - f(\eta_j x) + f(x) \right)^2 \\
&& + \sum_{i=1}^d \sum_{j=1}^d \left( f(\eta_j x) - f(x) \right) \left( f(\eta_j \eta_i x) - f(\eta_i x) - f(\eta_j x) + f(x) \right).
\end{eqnarray*} 
On the other hand, we have for the second term on the RHS of \eqref{eq:Gamma2}, using Ricci flatness,
\begin{eqnarray*}
- 2 \Gamma(f,\Delta f)(x) &=& - \sum_{j=1}^d \left( f(\eta_j x) -f(x) \right) \left( \Delta f(\eta_j x) - \Delta f(x) \right) \\
&=& - \sum_{j=1}^d \sum_{i=1}^d  \left( f(\eta_j x) -f(x) \right) \left( f(\eta_i \eta_j x) - f(\eta_j x) - f(\eta_i x) + f(x)  \right) \\
&=& - \sum_{j=1}^d \sum_{i=1}^d  \left( f(\eta_j x) -f(x) \right) \left( f(\eta_j \eta_i x) - f(\eta_i x) - f(\eta_j x) + f(x)  \right).
\end{eqnarray*}
Adding both terms, we end up with
$$ 2 \Gamma_2(f,f)(x) = \sum_{i=1}^d \sum_{j=1}^d \left( f(\eta_j \eta_i x) - f(\eta_i x) - f(\eta_j x) + f(x) \right)^2  \ge 0, $$
showing $\K_\infty(x) \ge 0$. Under the stronger condition of (R)-Ricci flatness, we can estimate $2 \Gamma_2(f,f)(x)$
from below as follows:
\begin{eqnarray*}
2 \Gamma_2(f,f)(x) &=& \sum_{i=1}^d \sum_{j=1}^d \left( f(\eta_j \eta_i x) - f(\eta_i x) - f(\eta_j x) + f(x) \right)^2 \\
&\ge& \sum_{i=1}^d \left( f(\eta_i \eta_i x) - f(\eta_i x) - f(\eta_i x) + f(x) \right)^2 \\
&=& \sum_{i=1}^d \left( 2 f(x) - 2 f(\eta_i x) \right)^2 = 4 \Gamma(f,f)(x).
\end{eqnarray*}
This shows that $\Gamma_2(f,f)(x) \ge 2 \Gamma(f,f)(x)$, which means that we have $\K_\infty(x) \ge 2$.
\end{proof}

\section{Triangle-free graphs}
\label{sec:triangle-free}

In this section we focus on curvature comparison results for
graphs without triangles. Our main result states that non-negativity of Ollivier
Ricci curvature implies non-negativity of Bakry-\'Emery curvature
under a certain in-degree condition (see Corollary
\ref{cor:curvimp}). This result is derived via Ricci flatness
properties.
 
We start with particular upper curvature bounds in case of triangle-freeness:

\begin{prop} \label{prop:uppbdcurv}
  Let $G=(V,E)$ be $d$-regular. Then we have the
  following upper curvature bounds:
  \begin{itemize}
  \item[(i)] $\kappa_0(x,y) \le 0$ for all edges $\{x,y\} \in E$ not contained in a triangle,
  \item[(ii)] $\kappa_{LLY}(x,y) \le \frac{2}{d}$ for all edges
    $\{x,y\} \in E$ not contained in a triangle,
  \item[(iii)] $\K_\infty(x) \le 2$ for all $x \in V$ not contained in a triangle.
  \end{itemize}
\end{prop}

\begin{remark}
  Combining the proposition with the lower curvature bounds for Ricci
  flatness (Theorems \ref{thm:ORcurvRF} and \ref{thm:BEcurvRF}), we
  obtain the following curvature equalities:
  \begin{itemize}
  \item If $x$ is Ricci flat and the egde $\{x,y\} \in E$
    is not contained in any triangle then $\kappa_0(x,y) = 0$.
  \item If $x$ is $(R)$-Ricci flat and the egde $\{x,y\} \in E$
    is not contained in any triangle then $\kappa_{LLY}(x,y) = \frac{2}{d}$.
  \item If $x$ is $(R)$-Ricci flat and not contained in any triangle
    then $\K_\infty(x) = 2$.
  \end{itemize}
\end{remark}

\begin{proof}[Proof of Proposition \ref{prop:uppbdcurv}]
	
	Although Statements (i) and (ii) are an implication from Theorem \ref{thm:uppbdLLY}, we provide their proof here which presents a useful idea for the following remark.
	\medskip
	
	Statement (i) follows from
	\begin{eqnarray} \label{eq:k0upbd}
	W_1(\mu_x^0,\mu_y^0) &=& \sum_{u \in S_1(x)} \sum_{v \in S_1(y)}
	d(u,v) \pi_{opt}(u,v)\nonumber \\ &\ge& \sum_{u \in S_1(x)} \sum_{v \in S_1(y)}
	\pi_{opt}(u,v) = 1, 
	\end{eqnarray}
	since $S_1(x) \cap S_1(y) = \emptyset$. Here $\pi_{opt}$ is an
	optimal transport plan in $\Pi(\mu_x^0,\mu_y^0)$.
	
	For the proof of (ii), we only need to show 
	$$ \kappa_{\frac{1}{d+1}}(x,y) \le \frac{2}{d+1}, $$
	by \eqref{eq:LLYOllconn}. This follows from
	\begin{eqnarray} \label{eq:kdpoupbd}
	W_1(\mu_x^{1/(d+1)},\mu_y^{1/(d+1)}) &=& \sum_{u \in B_1(x)}
	\sum_{v \in B_1(y)} d(u,v) \pi_{opt}(u,v)\nonumber \\ &\ge& \left(
	\sum_{u \in B_1(x)} \sum_{v \in B_1(y)} \pi_{opt}(u,v) \right) -
	\pi_{opt}(x,x) - \pi_{opt}(y,y)\nonumber \\ &\ge& 1 -
	\frac{2}{d+1},
	\end{eqnarray}
	since $B_1(x) \cap B_1(y) = \{x,y\}$ and $\pi_{opt}(u,u) \le \mu_x^{1/(d+1)}(u)
	\le \frac{1}{d+1}$. Here $\pi_{opt}$ is an
	optimal transport plan in $\Pi(\mu_x^{1/(d+1)},\mu_y^{1/(d+1)})$.
	\medskip
	
	Statement (iii) is an implication from Theorem \ref{thm:uppbdBE}.
\end{proof}

\begin{remark} \label{rem:curvcomb}
  Note that in Proposition \ref{prop:uppbdcurv}, (ii) implies (i) by
  Theorem \ref{thm:kLLYk0comp}. Moreover, it follows from the above
  proof that sharpness of the bounds in (i) and (ii) has the
  following combinatorial interpretation in the triangle-free case:
  \begin{itemize}
  \item[(a)] $\kappa_0(x,y)=0$ is equivalent that there is a perfect matching
    between $S_1(x)$ and $S_1(y)$.
  \item[(b)] $\kappa_{LLY}(x,y)= \frac{2}{d}$ is equivalent that there is a
    perfect matching between $S_1(x) \backslash \{y\}$ and
    $S_1(y)\backslash \{x\}$.
  \end{itemize}
\end{remark}

A natural class of examples where all three upper bounds of
Proposition \ref{prop:uppbdcurv} are attained are distance-regular
graphs of girth $4$ (see Section \ref{sec:distreg} below). To motivate
our next result, let us focus on one particular example:

\begin{example} Let $S_1(x) = \{v_1,\dots,v_d\}$ and
  $S_2(x) = \{v_{ij} \mid 1 \le i < j \le d \}$ with
  $v_i,v_j \sim v_{ij}$. In fact this is the $2$-ball of the
  $d$-dimensional hypercube $Q^d$ and we have the following curvatures
  (see Remark \ref{rem:excurv}):
  $$ \kappa_0(x,v_i) = 0, \quad \kappa_{LLY}(x,v_i) = \frac{2}{d}, \quad
  \K_\infty(x) = 2. $$
  We also like to mention that the vertex $x$ in this example is
  $(RS)$-Ricci flat and that we have $d_x^-(z) = 2$ for all
  $z \in S_2(x)$.
\end{example}

\begin{thm} \label{thm:curvRic} Given a regular graph $G=(V,E)$, let
  $x \in V$ be a vertex not contained in a triangle and satisfying
  $d_x^-(z) \le 2$ for all $z \in S_2(x)$. Then we have the following:
  \begin{itemize}
    \item[(a)] $\kappa_0(x,y) = 0$ for all $y \in S_1(x)$ is equivalent to $x$ being (S)-Ricci flat.
    \item[(b)] $\kappa_{LLY}(x,y) = \frac{2}{d}$ for all
      $y \in S_1(x)$ is equivalent to $x$ being $(RS)$-Ricci flat.
  \end{itemize}
\end{thm}

This result, together with Theorem \ref{thm:BEcurvRF}, implies our
main curvature comparison result in Theorem \ref{cor:curvimp} from the
Introduction:

\begin{proof}[Proof of Theorem \ref{cor:curvimp}]
  Under the assumptions of Theorem \ref{thm:curvRic}, we first assume
  that $\kappa_0(x,y) = 0$ for all $y \in S_1(x)$. This implies that
  $x$ is Ricci flat and, by Theorem \ref{thm:BEcurvRF}(a), that
  $\K_\infty(x) \ge 0$.

  Similarly, assuming $\kappa_{LLY}(x,y) = \frac{2}{d}$ for all
  $y \in S_1(x)$, we know that $x$ is $(R)$-Ricci flat, and Theorem
  \ref{thm:BEcurvRF}(b) implies that $\K_\infty(x) \ge 2$. Since
  $x$ is not contained in a triangle, this leads to $\K_\infty(x) = 2$
  by Proposition \ref{prop:uppbdcurv}(iii).
\end{proof}

Before we start with the proof of Theorem \ref{thm:curvRic}, let us
introduce the following notion and discuss relations to existing
results.

\begin{defn}
  Let $G=(V,E)$ be a regular triangle-free graph and $x \in V$. We say that $y_1,y_2 \in S_1(x)$ are \emph{linked}
  by $z \in S_2(x)$ if we have $y_1 \sim z \sim y_2$.  We refer to $z$ as a \emph{link} of $y_1$ and $y_2$.
\end{defn}

P. Ralli \cite{Ralli} investigated curvature implications for regular
graphs without $K_3$ and $K_{3,2}$ as subgraphs. It is easy to check
that this condition is equivalent to the following properties at all
vertices $x$:
  \begin{itemize}
  \item[(i)] $x$ is not contained in a triangle,
  \item[(ii)] $d_x^-(z) \le 2$ for all $z \in S_2(x)$,
  \item[(iii)] Any pair $y_1,y_2 \in S_1(x)$ has at most one link. 
  \end{itemize}
  A consequence of his results is that conditions (i),(ii),(iii) imply $\K_\infty(x) \le 0$ or $\K_\infty(x) = 2$. Under these conditions, Ralli has the following equivalence:
  $$ \text{$\kappa_0(x,y) = 0$ for all $y \in S_1(x)$} \Longleftrightarrow \text{$\K_\infty(x) \ge 0$}. $$
  Our theorem implies that the implication ''$\Longrightarrow$'' holds already under conditions (i) and (ii) and
  we have an example that the implication ''$\Longleftarrow$'' is no longer true if one drops condition (iii).  

\begin{proof}[Proof of Theorem \ref{thm:curvRic}]
  The implications $\Longleftarrow$ in (a) and (b) follow immediately from
  Theorem \ref{thm:ORcurvRF} and Proposition \ref{prop:uppbdcurv}.

  Let us now prove the forward implication in (a). Let $x \in V$ be given with $d = d_x$ and $S_1(x) = \{y_1,\dots,y_d\}$. The property $\kappa_0(x,y) = 0$ for
  all $y \in S_1(x)$ implies that we have perfect matchings $\sigma_i: S_1(x) \to S_1(y_i)$ for all $1 \le i \le d$. 
  In particular, we can assume that these perfect matchings $\sigma_i$ satisfy the following property:
  
  \medskip
  
  {\bf Property (P):} If  there exists a perfect matching between $S_1(x) \backslash \{ y_i \}$ and $S_1(y_i) \backslash \{ x \}$ then  $\sigma_i(y_i) = x$. 
  
  \medskip

  Our goal is to show that we can modify these perfect matchings in such a way that $\sigma_i(y_j) = \sigma_j(y_i)$
  for all $i \neq j$. Defining then $\eta_i: B_1(x) \to B_1(y_i)$ as $\eta_i(x) = y_i$ and $\eta_i(y) = \sigma_i(y)$ for $y \in S_1(x)$ provide (S)-Ricci flatness.
   
   We first prove the following crucial fact:
   
   \medskip
   
   {\bf Fact:} Let $i \neq j$. We have $\sigma_i(y_j) = x$ if and only if $y_i$ and $y_j$ are not linked.
   
   \medskip
   
   This fact can be shown as follows: We first prove the easier ''$\Longleftarrow$'' implication. Assume $y_i$ and $y_j$ are not linked. Then $\sigma_i(y_j) \sim y_i,y_j$ cannot be in $S_2(x)$ and we must have therefore $\sigma_i(y_j) = x$.
   For the ''$\Longrightarrow$'' implication, we provide an indirect proof: If $y_i$ and $y_j$ were linked by $z \in S_2(x)$, then the $\sigma_i$-preimage of $z \in S_1(y_i)$ must be in $\{y_i,y_j\}$ but we know that $\sigma_i(y_j) = x$. Therefore $\sigma_i(y_i)=y_j$.
   Defining then the map $\tilde \sigma_i: S_1(x) \to S_1(y_i)$ via
   $$ \tilde \sigma_i(y_k) = \begin{cases} \sigma_i(y_k) & \text{if $k \neq i,j$,}\\ 
   z & \text{if $k=j$,}\\ x & \text{if $k=i$,} \end{cases} $$
   induces a perfect matching between $S_1(x) \backslash \{ y_i \}$ and $S_1(y_i) \backslash \{ x \}$. This would imply
   $\sigma_i(y_i) = x$ contradicting to $\sigma_i(y_j) = x$.
   
   \medskip
   
   Now we prove our goal.
   
   We first show that $\sigma_i(y_j) = x$ implies $\sigma_j(y_i) = x$: Since $\sigma_i(y_j) = x$, $y_i$ and $y_j$ are not linked by our Fact which, in turn, implies $\sigma_j(y_i) = x$ by our Fact, again.
   
   We deal with all other pairs $(i,j)$, $i \neq j$ as follows: If $\sigma_i(y_j) = \sigma_j(y_i)$, we do not change the assignments $\sigma_i(y_i),\sigma_i(y_j),\sigma_j(y_i),\sigma_j(y_j)$. Now we assume that $\sigma_i(y_j) =: z \neq
   \sigma_j(y_i) := z'$. Note that $z,z' \in S_2(x)$ and they both are links of $y_i$ and $y_j$. Since $z \in S_1(y_j)$ and
   $d_x^{-}(z) \le 2$, we must have $\sigma_j^{-1}(z) \in \{ y_i,y_j \}$. Since $\sigma_j$ is injective and $\sigma_j(y_i)=z'$, we must have $\sigma_j^{-1}(z) = y_j$. So we must have
\begin{equation} \label{eq:sigmaz}
   \sigma_j(y_j) = z.
\end{equation}
Similarly, we conclude that $\sigma_i(y_i) = z'$. Now we modify $\sigma_i$ as follows: $\sigma_i(y_i) = z$ and $\sigma_i(y_j) = z'$. This preserves property (P) of the perfect matching $\sigma_i$ and establishes $\sigma_i(y_j) = \sigma_j(y_i)$ for this pair of indices $(i,j)$. Note that if $(i,j)$ and $(k,l)$ are two different pairs with $\sigma_i(y_j) \neq \sigma_j(y_i)$ and $\sigma_k(y_l) \neq \sigma_l(y_k)$ then 
   $\{i,j\} \cap \{k,l\} = \emptyset$ for, otherwise, if $k=i$, there is no perfect matching between $S_1(x)$ and $S_1(y_i)$
   since the four links between $y_i,y_j$ and $y_i,y_l$ can only have three possible preimages under $\sigma_i$.
This guarantees that we can repeat this process for all such pairs $(i,j)$ simultaneously and we will end up with the    
required symmetric arrangement.

Finally, it remains to prove the forward implication of (b). The
assumption $\kappa_{LLY}(x,y) = \frac{2}{d}$ for all $y \in S_1(x)$
implies $\kappa_0(x,y) = 0$ by Theorem \ref{thm:kLLYk0comp}. The
existence of perfect matchings between $S_1(x)\backslash\{y_i\}$ and
$S_1(y_i)\backslash\{x\}$ for all $1 \le i \le d$ from Remark
\ref{rem:curvcomb} further imply that our chosen maps $\sigma_i$
satisfy $\sigma_i(y_i)=x$ for all $i$. In this situation, we can
disregard the above possibility of
$z = \sigma_i(y_j) \neq \sigma_j(y_i) = z'$ with $z,z' \in S_2(x)$,
since this would imply \eqref{eq:sigmaz}, which contradicts to
$\sigma_j(y_j) =x$.  Therefore, the maps $\sigma_i$ do not need to be
modified and the induced maps $\eta_i: B_1(x) \to V$ satisfy both
symmetry and reflexivity.
\end{proof}

\begin{remark}\ \\
  (a) The reverse of the implication in Theorem \ref{cor:curvimp}(a)
  is not true since we have a triangle-free $2$-ball in
  Figure \ref{fig:count_ex} with $\K_\infty(x) = 0$, $d_x^-(z) = 2$
  for all $z \in S_2(x)$ and $\kappa_0(x,y) < 0$ for all
  $y \in S_1(x)$ as a counterexample. Note that
  $S_1(x) = \{v_1,\dots,v_6\}$.
  
  \tikzstyle{every node}=[circle, draw, fill=black,
                        inner sep=0pt, minimum width=4pt]
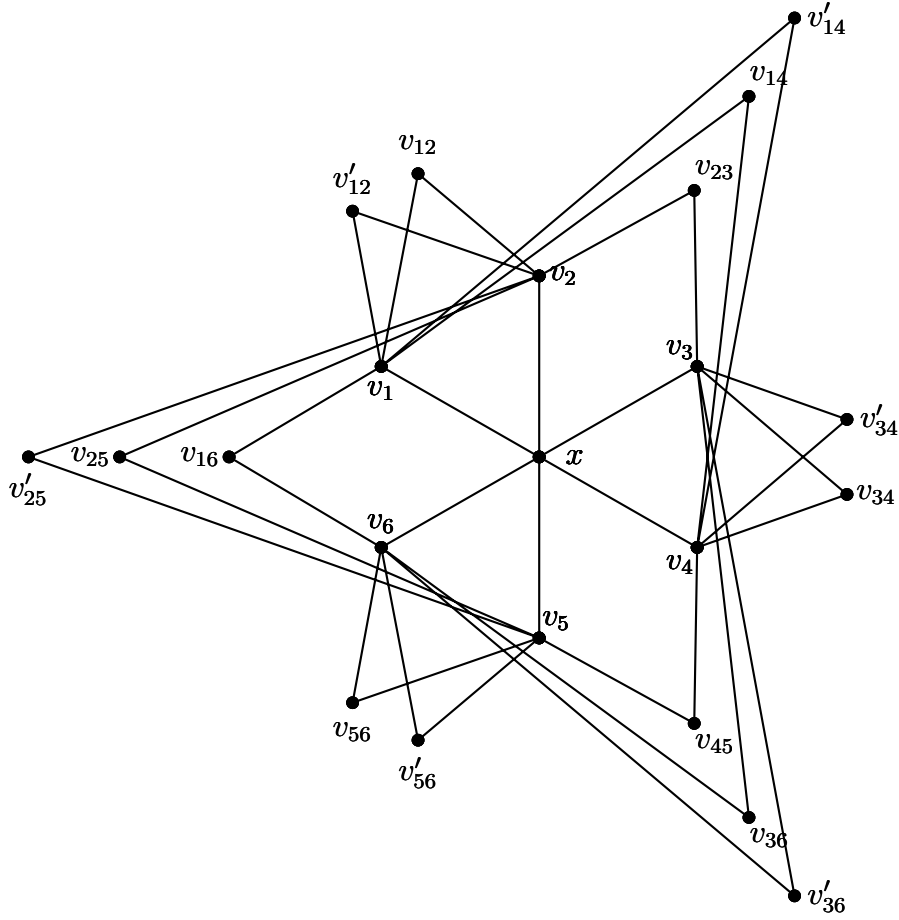
\begin{figure}[h!]
\begin{center}
\begin{tikzpicture}[thick,scale=1.2]%
    \draw \foreach \y in {30,90,...,330} {
		(\y:2) -- (0:0) node[label=right:\;\;$x$]{}    	  
		(150:2) node[label=below:$v_1$]{}
		(210:2) node[label=above:$v_6$]{}
		(270:2) node[label=above right:$v_5$]{}
		(330:2) node[label=below left:$v_4$]{}
		(30:2) node[label=above left:$v_3$]{}
		(90:2) node[label=right:$v_2$]{}
};
   \draw \foreach \z in {60,180,300} {
 	   (\z:3.4) -- (\z-30:2)
 	   (\z:3.4) -- (\z+30:2)
 	   (\z:4.6) -- (\z-90:2)
 	   (\z:4.6) -- (\z+90:2)
 	   (\z:5.6) -- (\z-90:2)
 	   (\z:5.6) -- (\z+90:2)
 	   (\z+53:3.4) -- (\z+30:2)
 	   (\z+67:3.4) -- (\z+30:2)
 	   (\z+53:3.4) -- (\z+90:2)
 	   (\z+67:3.4) -- (\z+90:2)
 	   (60:3.4) node[label=above right:{$v_{23}$}]{}
 	   (180:3.4) node[label=left:{$v_{16}$}]{}
 	   (300:3.4) node[label=below right:{$v_{45}$}]{}
 	   (60:4.6) node[label=above right:{$v_{14}$}]{}
 	   (180:4.6) node[label=left:{$v_{25}$}]{}
 	   (300:4.6) node[label=below right:{$v_{36}$}]{}
 	   (60:5.6) node[label=right:{$v_{14}'$}]{}
 	   (180:5.6) node[label=below:{$v_{25}'$}]{}
 	   (300:5.6) node[label=right:{$v_{36}'$}]{}
 	   (113:3.4) node[label=above:{$v_{12}$}]{}
 	   (127:3.4) node[label=above:{$v_{12}'$}]{}
 	   (233:3.4) node[label=below:{$v_{56}$}]{}
 	   (247:3.4) node[label=below:{$v_{56}'$}]{}
 	   (353:3.4) node[label=right:{$v_{34}$}]{}
 	   (7:3.4) node[label=right:{$v_{34}'$}]{}
   };
\end{tikzpicture}
\end{center}
\caption{Example with $\K_\infty(x) = 0$, $d_x^-(z) = 2$ for all
  $z \in S_2(x)$ and $\kappa_0(x,v_i) = - \frac{1}{3}$.
  \label{fig:count_ex}}
\end{figure}

\smallskip

  \noindent
  (b) All conditions in Theorem \ref{thm:curvRic}(a) are necessary: 
  \begin{itemize}
  \item[(i)] If $x$ is contained in a triangle, we have the
    icosidodecahedral graph (see Figure \ref{fig:icosidodecahedron})
    as a counterexample with $\kappa_0(x,y)=0$ for all edges $\{x,y\}$
    but $\K_\infty(x) < 0$ for all vertices $x$, which means that $x$
    cannot be Ricci flat by Theorem \ref{thm:BEcurvRF}.
  \item[(ii)] If we drop $d_x^-(z) \le 2$ for all $z \in S_2(x)$,
    Figure \ref{fig:counterexample} provides a counterexample with
    $\kappa_0(x,y) = 0$ for all $y \in S_1(x)$ and $\K_\infty(x) < 0$.
  \end{itemize}
  
\begin{figure}[h!]
\begin{center}
\tikzstyle{every node}=[circle, draw, fill=black!50,
                        inner sep=0pt, minimum width=4pt]
    
 \begin{tikzpicture}[thick,scale=1]%
        \draw (0,0) node(x1)[label=above:$x$] {};
        \draw (-2,-2) node(x2){};
        \draw (-1,-2) node(x3){};
        \draw (0,-2) node(x4){};
        \draw (1,-2) node(x5){};
        \draw (2,-2) node(x6){};
        \draw (-5,-4) node(x7){};
        \draw (-4,-4) node(x8){};
        \draw (-3,-4) node(x9){};
        \draw (-2,-4) node(x10){}; 
        \draw (-1,-4) node(x11){};
        \draw (0,-4) node(x12){};
        \draw (1,-4) node(x13){};
        \draw (2,-4) node(x14){}; 
        \draw (3,-4) node(x15){};
        \draw (4,-4) node(x16){};
        \draw (5,-4) node(x17){};
        \draw (x2)--(x1)--(x3)  (x4)--(x1)--(x5) (x1)--(x6);
        \draw (x2)--(x7) (x2)--(x8) (x2)--(x10) (x2)--(x11);
        \draw (x3)--(x9) (x3)--(x8) (x3)--(x10) (x3)--(x13);
        \draw (x4)--(x10) (x4)--(x11) (x4)--(x12) (x4)--(x15);
        \draw (x5)--(x11) (x5)--(x13) (x5)--(x14) (x5)--(x16);
        \draw (x6)--(x13) (x6)--(x15) (x6)--(x16) (x6)--(x17);
\end{tikzpicture}
\caption{Example with $\K_\infty(x)=-0.194 < 0$ and $\kappa_p(x,y)=0$ $\forall$
  $p\in[0,1]$, $y\sim x$.}
   \label{fig:counterexample}
   \end{center}
\end{figure}
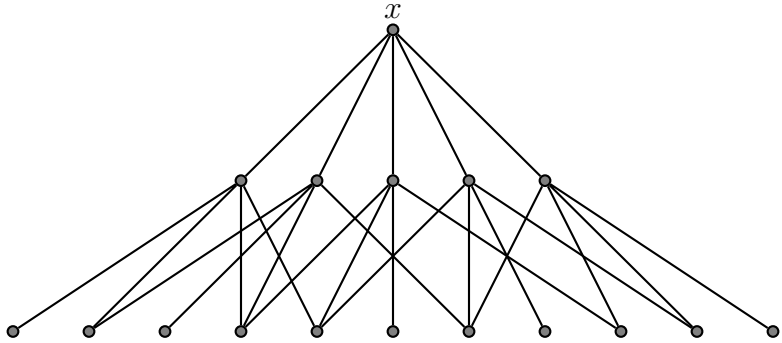

\noindent
(c) All conditions in Theorem \ref{thm:curvRic}(b) are necessary. Since
in the case of triangles we have the following upper bound
$$ \kappa_{LLY}(x,y) \le \frac{2+\#_\Delta(x,y)}{d}, $$
a natural generalization of the equivalence in the case of triangles
would be the following statement:
$$ \text{$\kappa_{LLY}(x,y) = \frac{2 + \#_\Delta(x,y)}{d}$ for all
$y \in S_1(x)$ is equivalent to $x$ being (RS)-Ricci flat.} $$
  \begin{itemize}
  \item[(i)] If $x$ is contained in a triangle, we have
    $K_3 \times K_3$ with $d=4$ as a
    counterexample:
    $$ \kappa_{LLY}(x,y) = \frac{3}{4} = \frac{2 + \#_\Delta(x,y)}{d} $$
    for all edges $\{x,y\}$, but no vertex of $K_3 \times K_3$ is $(RS)$-Ricci
    flat.
  \item[(ii)] If we drop $d_x^-(z) \le 2$ for all $z \in S_2(x)$, the
    $6$-regular incidence graph of the $(11,6,3)$-design provides a
    counterexample with $\kappa_{LLY}(x,y) = \frac{1}{3}$ for all
    $y \in S_1(x)$, but $x$ is not $(RS)$-Ricci flat (see Example \ref{ex:incidence_graph}).
  \end{itemize}
\end{remark}

\section{Graph products} 

This section is concerned with three natural products of two graphs
$G$ and $H$: the tensor product $G \otimes H$, the Cartesian product
$G \times H$, and the strong product $G \boxtimes H$. We will see that
Ricci flatness is preserved under all three products. However, while
Cartesian products preserve non-negativity of both Bakry-\'Emery and
Ollivier Ricci curvature, we will see that this property fails to be
true in the case of strong products.

Let us start with the definitions of these graph products:

\begin{defn} \label{def:graph-products}
  Let $G=(V_G,E_G)$ and $H=(V_H,E_H)$ be two graphs. The vertex set of
  each of the three products $G \otimes H$ (\emph{tensor product}),
  $G \times H$ (\emph{Cartesian product}) and $G \boxtimes H$
  (\emph{strong product}) is given by $V_G \times V_H$. To define
  the edge sets for each of these products, let
  \begin{eqnarray*}
    E_{\rm hor} &:=& \{ \{(x_1,y),(x_2,y)\} \mid x_1 \stackrel{G}{\sim} x_2 \}, \\
    E_{\rm vert} &:=& \{ \{(x,y_1),(x,y_2)\} \mid y_1 \stackrel{H}{\sim} y_2 \}, \\
    E_{\rm diag} &:=& \{ \{(x_1,y_1),(x_2,y_2)\} \mid \text{$x_1 \stackrel{G}{\sim} x_2$ and $y_1 \stackrel{H}{\sim} y_2$} \}
  \end{eqnarray*}
  denote the set of \emph{horizontal}, \emph{vertical} and
  \emph{diagonal} edges. Then
  \begin{eqnarray*}
    G \otimes H &:=& (V_G \times V_H, E_{\rm diag}), \\
    G \times H &:=& (V_G \times V_H, E_{\rm hor} \cup E_{\rm vert} ), \\
    G \boxtimes H &:=& (V_G \times V_H, E_{\rm hor} \cup E_{\rm vert} \cup E_{\rm diag}).
  \end{eqnarray*}
\end{defn}

Note that, in the case of a $d_G$-regular graph $G$ and a
$d_H$-regular graph $H$, the products $G \otimes H$, $G \times H$ and
$G \boxtimes H$ are $(d_G d_H)$-regular, $(d_G + d_H)$-regular and
$(d_G + d_H + d_g d_H)$-regular, respectively.

Our first result is concerned with preservance of Ricci flatness:

\begin{thm} \label{thm:preserv-Rf}
  Let $G,H$ be two Ricci flat graphs. Then the graph products $G \otimes H$, $G \times H$ and $G \boxtimes H$
  are again Ricci flat. Similarly, all three graph products preserve also $(R)$-Ricci flatness, $(S)$-Ricci flatness and $(RS)$-Ricci flatness.
\end{thm}

\begin{proof}
  Assume that $G$ and $H$ are Ricci flat at $x\in V_G$ and at
  $y\in V_H$, respectively, that is, there exist maps
  $\eta^G_i: B_1(x)\rightarrow V_G$ ($1\le i\le d_G$) and
  $\eta^H_k: B_1(y) \rightarrow V_H$ ($1\le k\le d_H$) satisfying the
  conditions (i),(ii),(iii) in Definition \ref{defn:ricci_flat}.

  Note that we have the inclusions
  $$ B_1^{G \times H}(x,y), B_1^{G \otimes H}(x,y) \subset B_1^{G \boxtimes H}(x,y). $$
  We define the following maps
  $\eta_i', \eta_k'', \eta_{jl}^\otimes: B_1^{G \boxtimes H}(x,y) \to
  V_G \times V_H$ (for $1 \le i,j \le d_G$, $1 \le k,l \le d_H$):
  \begin{eqnarray*}
    \eta'_{i}(u,v)&:=&(\eta^G_i(u), v),\\
    \eta''_{k}(u,v)&:=&(u, \eta^H_k(v)),\\
    \eta_{jl}^\otimes(u,v)&:=&(\eta^G_j(u),\eta^H_l(v)).
  \end{eqnarray*}
  Note that
  $$ \eta_{jl}^\otimes = \eta'_j \circ \eta''_l = \eta''_l \circ \eta'_j. $$
  We only consider the strong product case here, since all other
  products can be dealt with similarly by restrictions of the relevant
  $\eta$-maps to the corresponding $1$-balls. We now check properties
  (i), (ii) and (iii) of Definition \ref{defn:ricci_flat} for these maps
  on $B_1^{G \boxtimes H}(x,y)$.

  \smallskip
  
  To verify (i), we observe that $(u,v) \sim \eta'_{i}(u,v)$
  represents a horizontal edge in $G \boxtimes H$,
  $(u,v) \sim \eta''_{k}(u,v)$ represents a vertical edge and
  $(u,v) \sim \eta_{jl}^\otimes(u,v)$ represents a diagonal edge.

  \smallskip

  Next, we verify (ii): The above observation implies that
  $\eta_i'(u,v), \eta_k''(u,v)$ and $\eta_{jl}^\otimes(u,v)$ are
  mutually distinct for any choices of $i,j,k,l$.  Moreover, it is
  easy to check that
  $$ \eta_i'(u,v) \neq \eta_j'(u,v), \quad \eta_k''(u,v) \neq \eta_l''(u,v),
  \quad \eta_{ik}^\otimes(u,v) \neq \eta_{jl}^\otimes(u,v) $$
  for any choice of $i \neq j$ and $k \neq l$.

  \smallskip

  Now we verify (iii): We have
  \begin{eqnarray*}
    \bigcup_{j,l} \eta^\otimes_{jl}(\eta^\otimes_{ik}(x,y))
    &=&\bigcup_{j} \eta^G_j \eta^G_i x \times  \bigcup_{l} \eta^H_l \eta^H_k y\\
    &=& \bigcup_{j} \eta^G_i \eta^G_j x \times  \bigcup_{l} \eta^H_k \eta^H_l y \\
    &=& \bigcup_{j,l} \eta^\otimes_{ik}(\eta^\otimes_{jl}(x,y)).
  \end{eqnarray*}
  Similar commutation properties holds for the other families of
  $\eta$-maps, that is, we have
  $$ \bigcup_* \eta^* \eta^{**} (x,y) = \bigcup_* \eta^{**} \eta^* (x,y) $$
  where $\eta^*$ and $\eta^{**}$ are maps within the families
  $\eta'_i$, $\eta''_k$ and $\eta^\otimes_{jl}$. Combining these
  results, we obtain
   \begin{eqnarray*}
    \bigcup_{j} \eta'_j(\eta^{\otimes}_{ik}(x,y)) \cup \bigcup_{l} \eta''_l(\eta^{\otimes}_{ik}(x,y)) \cup \bigcup_{j,l} \eta^\otimes_{jl}(\eta^\otimes_{ik}(x,y)) \\
    = \bigcup_{j} \eta^{\otimes}_{ik}(\eta'_j(x,y)) \cup \bigcup_{j} \eta^{\otimes}_{ik}(\eta'_j(x,y)) \cup \bigcup_{j,l} \eta^\otimes_{ik}(\eta^\otimes_{jl}(x,y))
   \end{eqnarray*}
   and
  \begin{eqnarray*}
    \bigcup_{j} \eta'_j(\eta'_{i}(x,y)) \cup \bigcup_{l} \eta''_l(\eta'_{i}(x,y)) \cup \bigcup_{j,l} \eta^\otimes_{jl}(\eta'_{i}(x,y)) \\
    = \bigcup_{j} \eta'_{i}(\eta'_j(x,y)) \cup \bigcup_{l} \eta'_{i}(\eta''_l(x,y)) \cup \bigcup_{j,l} \eta'_{i}(\eta^\otimes_{jl}(x,y))
  \end{eqnarray*}
  and
  \begin{eqnarray*}
    \bigcup_{j} \eta'_j(\eta''_{k}(x,y)) \cup \bigcup_{l} \eta''_l(\eta''_{k}(x,y)) \cup \bigcup_{j,l} \eta^\otimes_{jl}(\eta''_{k}(x,y)) \\
    = \bigcup_{j} \eta''_{k}(\eta'_j(x,y)) \cup \bigcup_{l} \eta''_{k}(\eta''_l(x,y)) \cup \bigcup_{j,l} \eta''_{k}(\eta^\otimes_{jl}(x,y)).
  \end{eqnarray*}

  \smallskip
  
In conclusion, Ricci flatness is preserved for all three graph products.
  
  \smallskip

  Finally, we verify preservance of $(R)$-, $(S)$- and $(RS)$-Ricci
  flatness. Assume $(R)$-Ricci flatness at $x \in V_G$ and
  $y \in V_h$. $(R)$-Ricci flatness at $(x,y)$ follows now from
  $$ \left(\eta_{jl}^\otimes\right)^2(x,y) = (\left(\eta_j^G\right)^2(x),
  \left(\eta_l^H\right)^2(y)) = (x,y), $$
  and $(\eta_i')^2(x,y)=(\eta_k'')^2(x,y)=(x,y)$ can be checked
  similarly. Preservance of $(S)$-Ricci flatness follows from
  $$ \eta^* \eta^{**} (x,y) = \eta^{**} \eta^* (x,y) $$
  where $\eta^*$ and $\eta^{**}$ are maps within the families
  $\eta'_i$, $\eta''_k$ and $\eta^\otimes_{jl}$. 
\end{proof}

In the case of Cartesian products of two regular graphs $G,H$, there
are explicit curvature formulas in terms of curvatures of the factors:
Bakry-\'Emery curvature
$\K_\infty^{G \times H}(x,y) = \min\{ \K_\infty^G(x),\K_\infty^H(y) \}$ can
be found in \cite[Corollary 7.13]{CLP} and Ollivier
Ricci curvature $\kappa_0^{G \times H}(x,y)$ and
$\kappa_{LLY}^{G \times H}(x,y)$ can be found in \cite[Claim 1 and 2
in Proof of Theorem 3.1]{LLY}. In particular, non-negativity of each
  of these curvature notions is preserved under Cartesian products. In
  our next result, we provide lower curvature bounds for horizontal
  and vertical edges of the strong product $G \boxtimes H$:

  \begin{thm}\label{thm:strong_LLY}
    Let $G$ and $H$ be two regular graphs with vertex degrees $d_G$
    and $d_H$, respectively.
    Lower Ollivier Ricci curvature bounds on horizontal edges and vertical
    edges are given by
    \begin{eqnarray*}
      \kappa_{*}((x_1,y_1),(x_2,y_1)) &\ge&
      \frac{d_G(d_H+1)}{d_{G\boxtimes H}} \kappa_{*}^G(x_1,x_2), \\
      \kappa_{*}((x_1,y_1),(x_1,y_2)) &\ge&
      \frac{d_H(d_G+1)}{d_{G\boxtimes H}} \kappa_{*}^H(y_1,y_2),
    \end{eqnarray*}                                          
    where $\kappa_*$ may refer to $\kappa_0$ or $\kappa_{LLY}$ and $d_{G \boxtimes H} = d_G + d_H + d_G d_H$ is the vertex degree of
    $G \boxtimes H$.
\end{thm}

\begin{proof}[Proof of Theorem \ref{thm:strong_LLY}]
	
  Let us consider a horizontal edge $(x_1,y_1) \sim (x_2,y_1)$ where
  $x_1 \stackrel{G}{\sim}x_2$. We will prove this argument for Lin-Lu-Yau curvature first. Let
  $\pi_G\in \Pi(\mu_{x_1}^{1/(1+d_G)},\mu_{x_2}^{1/(1+d_G)})$ be an optimal transport
  plan, i.e., its cost is equal to $W_1^G(\mu_{x_1}^{1/(1+d_G)},\mu_{x_2}^{1/(1+d_G)})$.
  Now we define a function 
  $\pi: (V_G\times V_H)^2 \rightarrow [0,\infty)$ as follows:
  \begin{equation*}
  \pi\left((w_1,z_1),(w_2,z_2)\right) := 
  \begin{cases} \frac{1+d_G}{1+d_{G\boxtimes H}} \pi_G(w_1,w_2), & \text{if $z_1 = z_2 \in B_1(y_1)$,} \\
  0, & \text{otherwise.} \end{cases}
  \end{equation*}
  Now we verify the following marginal constraints showing that $\pi$ is indeed a transport plan $\pi \in \Pi(\mu_{(x_1,y_1)}^{1/(1+d_{G \boxtimes H})},\mu_{(x_2,y_2)}^{1/(1+d_{G \boxtimes H})})$: for fixed $(w_1,z_1) \in V_G \times V_H$,
  \begin{eqnarray*}
   \sum_{w_2,z_2} \pi((w_1,z_1),(w_2,z_2)) 
   &=& \frac{1+d_G}{1+d_{G\boxtimes H}}\cdot \1_{B_1(y_1)}(z_1) \sum_{w_2} \pi_G(w_1,w_2) \\ 
   &=& \frac{1+d_G}{1+d_{G\boxtimes H}} \cdot \1_{B_1(y_1)}(z_1) \cdot \mu_{x_1}^{1/(1+d_G)}(w_1) \\
   &=& \frac{1}{1+d_{G\boxtimes H}} \cdot \1_{B_1(y_1)}(z_1) \cdot \1_{B_1(x_1)}(w_1) \\
   &=& \frac{1}{1+d_{G\boxtimes H}} \cdot \1_{B_1(x_1,y_1)}(w_1,z_1) \\
   &=& \mu^{1/(1+d_{G \boxtimes H})}_{(x_1,y_1)}(w_1,z_1),
  \end{eqnarray*}
  and, similarly,
  $$ \sum_{w_1,z_1} \pi((w_1,z_1),(w_2,z_2)) = \mu^{1/(1+d_{G \boxtimes H})}_{(x_2,y_1)}(w_2,z_2). $$
 The cost of this transport plan can then be calculated as
\begin{eqnarray*}
{\rm cost}(\pi) 
&=& \sum_{(w_2,z_2)} \sum_{(w_1,z_1)} dist_{G\boxtimes H}\left((w_1,z_1),(w_2,z_2)\right)\pi((w_1,z_1),(w_2,z_2))\\
&=&  \sum_{z_1\in B_1(y_1)} \sum_{w_1,w_2} dist_{G}(w_1,w_2) \frac{1+d_G}{1+d_{G\boxtimes H}} \pi_G(w_1,w_2)\\
&=&\frac{(1+d_H)(1+d_G)}{1+d_{G \boxtimes H}}\sum_{w_1,w_2} dist_{G}(w_1,w_2)\pi_G(w_1,w_2)\\
&=& {\rm cost}(\pi_G).
\end{eqnarray*}
Recall that $\pi_G$ is assumed to be an optimal transport plan and, therefore,
\begin{multline*}
W_1^{G\boxtimes H}\left(\mu_{(x_1,y_1)}^{1/(1+d_{G\boxtimes H})},  \mu_{(x_2,y_1)}^{1/(1+d_{G\boxtimes H})}\right)\le {\rm cost}(\pi) \\ = {\rm cost}(\pi_G)=W_1^G(\mu_{x_1}^{1/(1+d_G)},\mu_{x_2}^{1/(1+d_G)}).
\end{multline*}
This inequality translates via Definition \ref{defn:Ollcurv} and
relation \eqref{eq:LLYOllconn} into:
$$\kappa_{LLY}((x_1,y_1),(x_2,y_1)) \ge \frac{d_G(d_H+1)}{d_{G\boxtimes H}} \kappa_{LLY}^G(x_1,x_2),$$
which gives the desired lower bound for $\kappa_{LLY}$ on the horizontal edge $(x_1,y_1) \sim (x_2,y_1)$. 

Now we prove a similar lower bound for $\kappa_0$. Let
$\pi^0_G\in \Pi(\mu_{x_1}^{0},\mu_{x_2}^{0})$ be an optimal transport
plan, whose cost is 
$${\rm cost}(\pi^0_G)=\sum_{\substack{w_1,w_2\in V_G \\{(w_1,w_2)\not=(x_1,x_2)}}} dist_{G}(w_1,w_2)  \pi^0_G(w_1,w_2),$$
where the condition $(w_1,w_2)\not=(x_1,x_2)$ on the summation can be imposed because $\pi^0_G(x_1,x_2)=0$ due to marginal constraints of $\pi^0_G$.

Define a function 
$\pi^0: (V_G\times V_H)^2 \rightarrow [0,\infty)$ as follows:
\begin{multline*}
\pi^0\left((w_1,z_1),(w_2,z_2)\right) \\ := 
\begin{cases} 
\frac{d_G}{d_{G\boxtimes H}} \pi^0_G(w_1,w_2), & \text{if $z_1 = z_2 \in B_1(y_1)$ and $(w_1,w_2)\not=(x_1,x_2)$, } \\
\frac{1}{d_{G\boxtimes H}}, & \text{if $z_1 = z_2 \in S_1(y_1)$ and $(w_1,w_2)=(x_1,x_2)$, } \\
0, & \text{otherwise.} \end{cases}
\end{multline*}
Now we verify that $\pi^0 \in \Pi(\mu_{(x_1,y_1)}^{0},\mu_{(x_2,y_1)}^{0})$:
Let $(w_1,z_1) \in V_G \times V_H$. We distinguish two cases:
\begin{enumerate}
\item If $w_1\not=x_1$ we have
\begin{eqnarray*}
	\sum_{w_2,z_2} \pi^0((w_1,z_1),(w_2,z_2))
	&=& \frac{d_G}{d_{G\boxtimes H}}\cdot \1_{B_1(y_1)}(z_1) \sum_{w_2} \pi^0_G(w_1,w_2) \\ 
	&=& \frac{d_G}{d_{G\boxtimes H}} \cdot \1_{B_1(y_1)}(z_1) \cdot \mu_{x_1}^{0}(w_1) \\
	&=& \frac{1}{d_{G\boxtimes H}} \cdot \1_{B_1(y_1)}(z_1) \cdot \1_{S_1(x_1)}(w_1)\\
	&=& \mu^{0}_{(x_1,y_1)}(w_1,z_1).
\end{eqnarray*}
The last equality follows from the fact that $w_1 \neq x_1$ implies
\begin{multline*}
\1_{B_1(y_1)}(z_1) \cdot \1_{S_1(x_1)}(w_1) = \1_{B_1(y_1)}(z_1) \cdot \1_{B_1(x_1)}(w_1) \\ = \1_{B_1(x_1,y_1)}(w_1,z_1) = \1_{S_1(x_1,y_1)}(w_1,z_1).
\end{multline*}
\item If $w_1=x_1$ we have
\begin{flalign*}
  \sum_{w_2,z_2} &\pi^0((x_1,z_1),(w_2,z_2)) \hspace*{8cm} \\
  &= \1_{S_1(y_1)}(z_1) \frac{1}{d_{G \boxtimes H}} + \1_{B_1(y_1)}(z_1) \frac{d_G}{d_{G \boxtimes H}} \sum_{w_2 \neq x_2} \underbrace{\pi_G^0(x_1,w_2)}_{=0} \\
  &= \frac{1}{d_{G \boxtimes H}} \1_{S_1(y_1)}(z_1) = \mu_{(x_1,y_1)}^0(x_1,z_1).
\end{flalign*}
\end{enumerate}
The verification of
$$ \sum_{w_1,z_1} \pi^0((w_1,z_1),(w_2,z_2)) = \mu_{(x_2,y_1)}^0(w_2,z_2) $$
is done similarly. The cost of $\pi^0$ can then be calculated as
\begin{eqnarray*}
	{\rm cost}(\pi^0) 
	&=& \sum_{(w_2,z_2)} \sum_{(w_1,z_1)} dist_{G\boxtimes H}\left((w_1,z_1),(w_2,z_2)\right)\pi^0((w_1,z_1),(w_2,z_2))\\
	&=&  \sum_{z_1\in B_1(y_1)} \sum_{(w_1,w_2)\not=(x_1,x_2)} dist_{G}(w_1,w_2) \frac{d_G}{d_{G \boxtimes H}} \pi^0_G(w_1,w_2)\\
	& & +\sum_{z_1\in S_1(y_1)}
	dist_{G}(x_1,x_2) \frac{1}{d_{G \boxtimes H}}\\
	&=& \frac{(1+d_H)d_G}{d_{G\boxtimes H}}{\rm cost}(\pi^0_G)+ \frac{d_H}{d_{G\boxtimes H}}.
\end{eqnarray*}

Therefore, we have
$$W_1^{G\boxtimes H}\left(\mu_{(x_1,y_1)}^{0},  \mu_{(x_2,y_1)}^{0}\right)\le {\rm cost}(\pi^0)=\frac{(1+d_H)d_G}{d_{G\boxtimes H}}W_1^G(\mu_{x_1}^{0},\mu_{x_2}^{0})+\frac{d_H}{d_{G\boxtimes H}},$$
or equivalently 
$$\kappa_{0}((x_1,y_1),(x_2,y_1)) \ge \frac{d_G(d_H+1)}{d_{G\boxtimes H}} \kappa_{0}^G(x_1,x_2),$$
which gives the desired lower bound for $\kappa_0$.

In the same way we obtain analogous results for vertical edges:
$$\kappa_{*}((x_1,y_1),(x_1,y_2)) \ge \frac{d_H(d_G+1)}{d_{G\boxtimes H}} \kappa_{*}^H(y_1,y_2).$$
\end{proof}

\begin{cor} \label{cor:strong_prod_nonneg}
  Let $G$ and $H$ be two regular graphs with non-negative $\kappa_0$
  (or $\kappa_{LLY}$). Then all horizontal and vertical edges of
  $G \boxtimes H$ have also non-negative $\kappa_0$ (or
  $\kappa_{LLY}$).
\end{cor}

It turns out, however, that the statement of Corollary
\ref{cor:strong_prod_nonneg} is no longer true for diagonal edges, as
the following example shows.

\begin{example}
  Let $G$ be a $4$-regular graph with an induced $2$-ball
  $B_2(v_0) = \{v_0,\dots,v_9\}$ as shown in Figure
  \ref{fig:strong_prod_counterexample}.  \begin{figure}[h!]
\begin{center}
  \includegraphics[width=\textwidth]{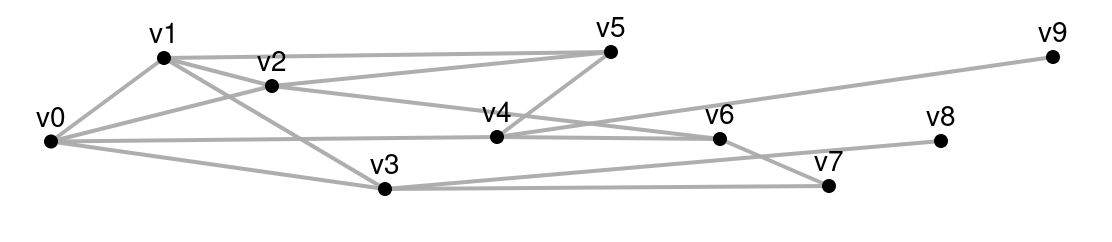}
\caption{Induced $2$-ball of a quartic graph with
  $\K_\infty(v_0) = 0.013$ and
  $\kappa_{LLY}(v_0,v_1) = 2 \kappa_0(v_0,v_1) = 1$,
  $\kappa_{LLY}(v_0,v_2) = 2 \kappa_0(v_0,v_2) = 0.5$ and
  $\kappa_{LLY}(v_0,v_3) = \kappa_0(v_0,v_3) = \kappa_{LLY}(v_0,v_4) =
  \kappa_0(v_0,v_4) = 0$.}
   \label{fig:strong_prod_counterexample}
   \end{center}
\end{figure}
 Then $\kappa_0(v_0,v_i) \ge 0$
  for $1 \le i \le 4$ and $\K_\infty(v_0) > 0$. Let
  $H = P_\infty$ be the bi-infinite paths with vertices $w_j$,
  $j \in \Z$. Then
  $\kappa_0(w_0,w_{\pm 1}) = \kappa_{LLY}(w_0,w_{\pm 1}) =0$ and
  $\K_\infty(w_0) = 0$.
  
\begin{figure}[h!]
\begin{center}
\begin{tikzpicture}[thick, scale=1]
\node[inner sep=0pt] (G) at (0,0)
{\includegraphics[width=\textwidth]{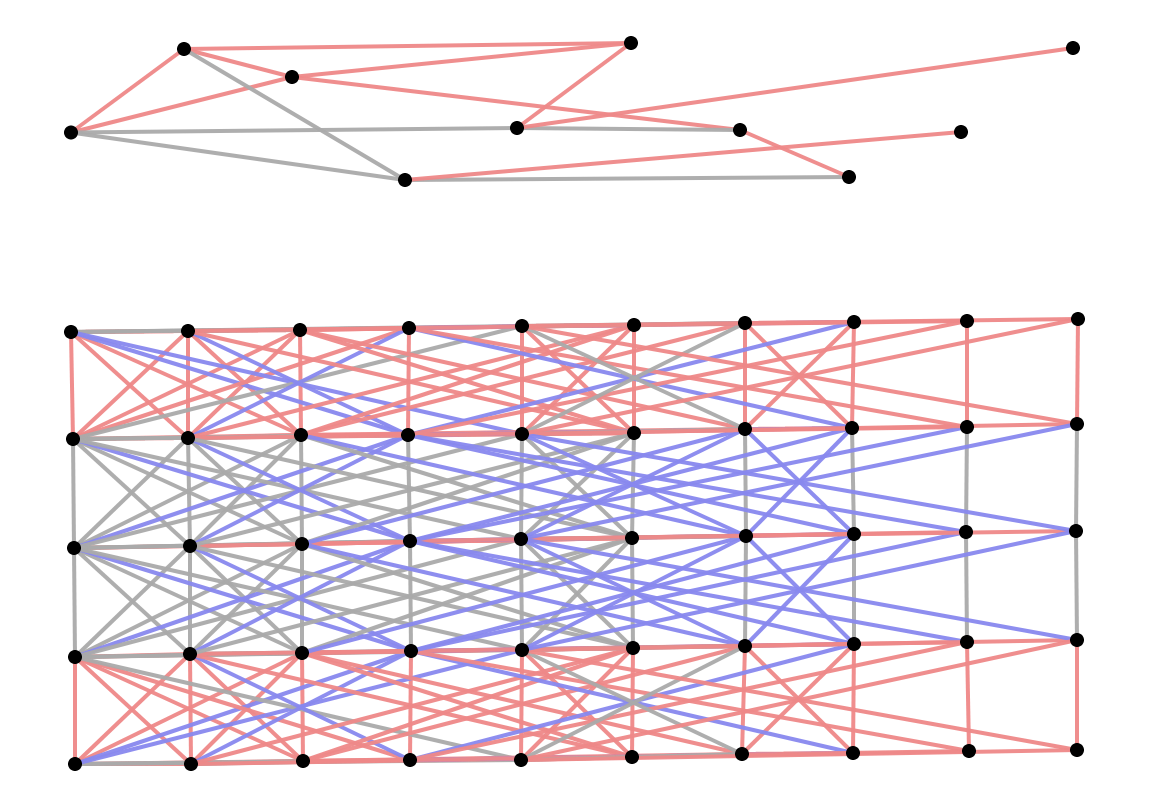}};

\draw (-5.55,3.23) node(v0) {$v_0$};
\draw (-4.3,4.15) node(v1) {$v_1$};
\draw (-3.12,3.84) node(v2) {$v_2$};
\draw (-1.87,2.71) node(v3) {$v_3$};
\draw (-0.65,3.26) node(v4) {$v_4$};
\draw (0.61,4.2) node(v5) {$v_5$};
\draw (1.80,3.25) node(v6) {$v_6$};
\draw (3.01,2.74) node(v7) {$v_7$};
\draw (4.23,3.25) node(v8) {$v_8$};
\draw (5.47,4.15) node(v9) {$v_9$};

\draw (-5.55,1.10) node(v02) {$(v_0,w_{2})$};
\draw (-5.55,-0.05) node(v01) {$(v_0,w_{1})$};
\draw (-5.55,-1.30) node(v00) {$\boldsymbol{(v_0,w_{0})}$};
\draw (-5.55,-2.50) node(v0-1) {$(v_0,w_{-1})$};
\draw (-5.55,-3.70) node(v0-2) {$(v_0,w_{-2})$};

\draw (-1.87,-0.05) node(v31) {$\boldsymbol{(v_3,w_{1})}$};
\end{tikzpicture}
	\caption{Local Ollivier Ricci curvatures $\kappa_{LLY}$ of $G$ and
    $G \boxtimes P_\infty$ at edges incident to $v_0$ and $(v_0,w_0)$,
    respectively. Positive/negative/zero curvatures of edges is represented by
    the colours red/blue/grey. Every horizontal line of the lower graph
    represents a projection of $G$.}
    \label{fig:strong_prod_lly}
\end{center}
\end{figure}
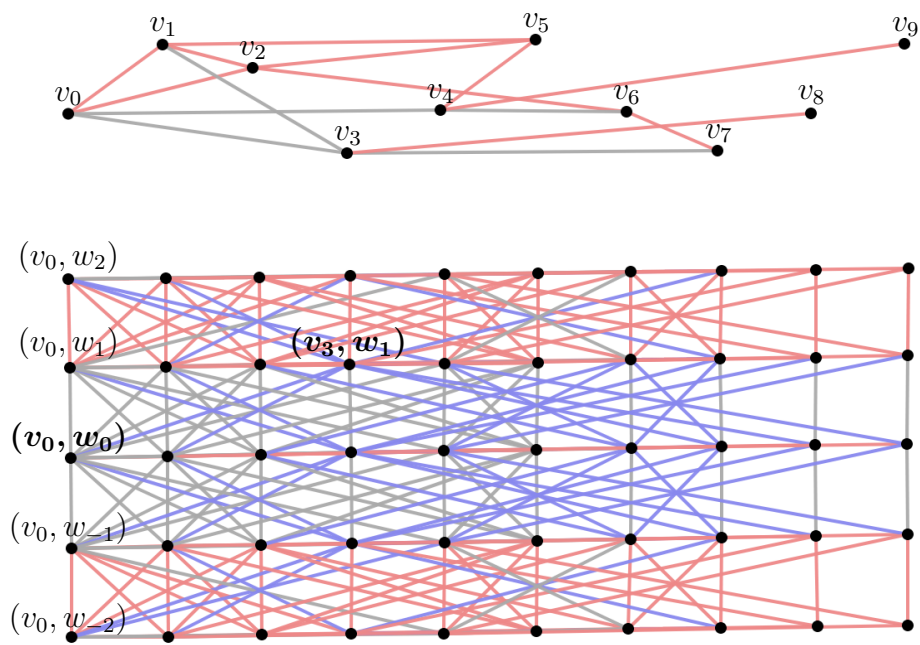

%
%
%

  However, the strong product $G \boxtimes H$ has negative Ollivier
  Ricci curvatures on the following diagonal edges (see Figure
  \ref{fig:strong_prod_lly}):
  $$ \kappa_0((v_0,w_0),(v_3,w_{\pm 1}) =
  \kappa_{LLY}((v_0,w_0),(v_3,w_{\pm 1}) = -0.071, $$ and negative
  Bakry-\'Emery curvature at $(v_0,w_0)$ (see Figure
  \ref{fig:strong_prod_BE}):
  $$ \K_\infty(v_0,w_0) = -0.062. $$

  \begin{figure}[h!]
\begin{center}
  \includegraphics[width=\textwidth]{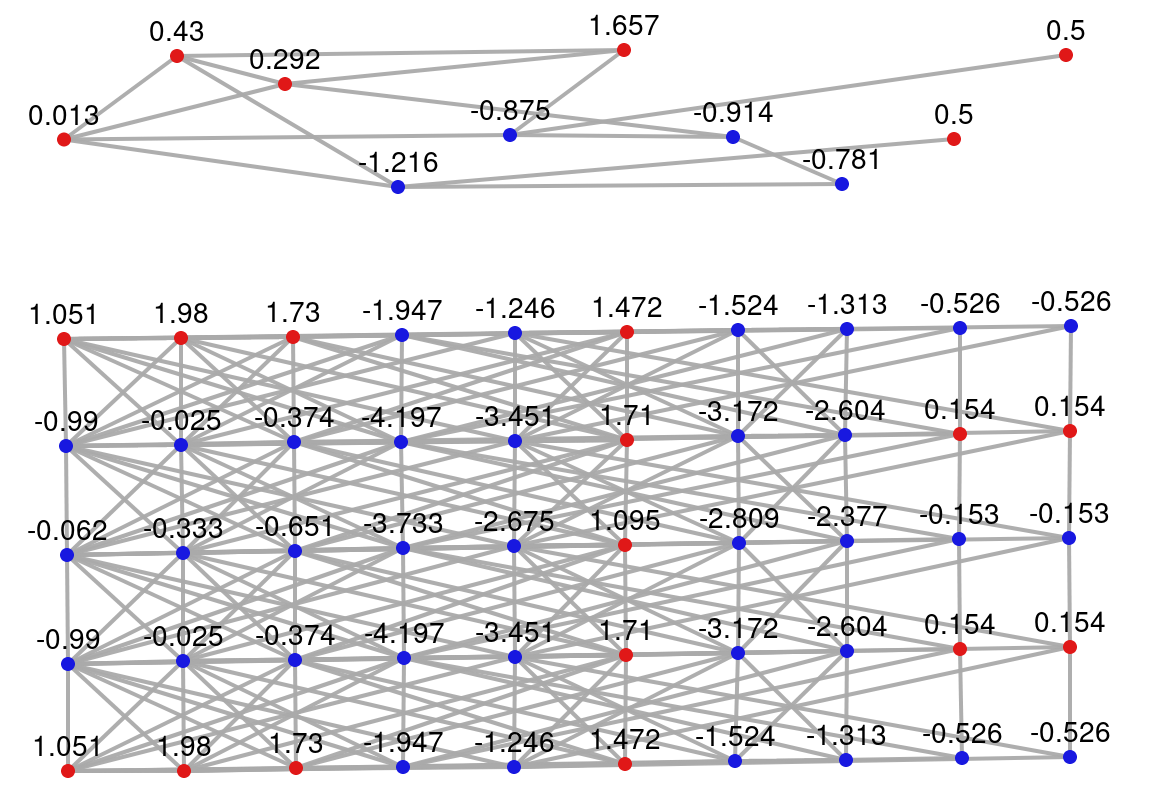}
  \caption{Local Bakry-\'Emery curvatures of $G$ and
    $G \boxtimes P_\infty$ at $v_0$ and
    $(v_0,w_0)$. Positive/negative curvatures of vertices is
    represented by the colours red/blue. Every horizontal line of
    the lower graph represents a projection of $G$.}
   \label{fig:strong_prod_BE}
   \end{center}
\end{figure}

\end{example}

\begin{remark} The previous example shows for strong products that
  non-negativity of curvatures is generally not preserved for diagonal
  edges. The same example can be used to show that this phenomenen
  appears also in the case of tensor products, where only diagonal
  edges are present.

  Another interesting question about graphs products is the following:
  In the case of Cartesian products, the full curvature function (as
  function of the dimension $\mathcal N$) at a vertex $(x,y)$ is
  completely determined by the curvature functions of the factors at
  the vertices $x$ and $y$ (see \cite[Theorem 7.9]{CLP}):
  $$ \K^{G \times H}_{(x,y)} = \K^G_x * \K^H_y, $$
  where $*$ is a special operation defined in \cite[Definition
  7.1]{CLP}. We would like to know whether a similar formula (with a
  suitably defined operation) can be proved for tensor products and
  strong products.
\end{remark}

\section{Distance-regular graphs}
\label{sec:distreg}

In this section we turn our focus on distance-regular graphs of
girth $4$, which is an interesting family of triangle-free graphs with
maximal curvature values for $\kappa_0,\kappa_{LLY}$ and
$\K_\infty$. Distance-regular graphs are defined as follows:

\begin{defn} A regular graph $G=(V,E)$ is called \emph{distance-regular}
  if, for any pair $x,y \in V$ of vertices and any $r,t \ge 0$
  the cardinality of $S_r(x) \cap S_t(y)$ depends only on $r,t,d(x,y)$. 

  The \emph{intersection array} of a distance-regular graph $G=(V,E)$
  of vertex degree $d$ is defined as an array of integers:
  $$ \{b_0,b_1,\dots,b_{d-1};c_1,\dots,c_d\}, $$
  defined as follows: Fix $x \in V$. Then, for $0 \le i \le d-1$ and
  $1 \le j \le d$, we set $b_i = d_x^+(z)$ for every $z \in S_i(x)$
  and $c_j = d_x^-(z)$ for every $z \in S_j(x)$.
\end{defn}

\begin{thm} \label{thm:distreg}
  Let $G=(V,E)$ be a distance-regular graph of vertex degree $d$ and
  girth $4$. Then we have
  \begin{equation} \label{eq:distregOR}
    \text{$\kappa_0(x,y) = 0$\,\, and\ $\kappa_{LLY}(x,y) = \frac{2}{d}$\,\, for all $\{x,y\} \in E$,}
  \end{equation}
  and 
  \begin{equation} \label{eq:distregBE}
    \text{$\K_\infty(x) = 2$\,\, for all $x \in V$.}
  \end{equation}
\end{thm}

Note that the curvature values in \eqref{eq:distregOR} and
\eqref{eq:distregBE} are upper curvature bounds for any triangle-free
$d$-regular graph by Proposition \ref{prop:uppbdcurv}.

Theorem \ref{thm:distreg} is a generalization of \cite[Theorem
4.10]{BCDDFP} and \cite[Corollary 11.7(i) in the arXiv
version]{CLP}, which are both concerned with the special case of
strongly regular graphs. Even though the proofs for this special
case carry over to the much larger class of distance-regular graphs,
we present them here for the reader's convenience.

\begin{proof} Let $G=(V,E)$ be a distance-regular graph of vertex
  degree $d$ and girth $4$ and $\{ x,y \} \in E$. By Remark
  \ref{rem:curvcomb}(b), it suffices to show the existence of a
  perfect matching between $S_1(x) \backslash \{ y \}$ and
  $S_1(y) \backslash \{ x \}$ to conclude
  \begin{equation} \label{eq:k_LLY-distreg}
  \kappa_{LLY}(x,y) = \frac{2}{d}.
  \end{equation}
  Let $H$ be the
  induced subgraph of the union of $S_1(x) \backslash \{ y \}$ and
  $S_1(y) \backslash \{ x \}$. Note that $H$ is bipartite since $G$ is
  triangle-free. Let $X \subset S_1(x) \backslash \{ y \}$ and $Y$
  be the set of neighbours of $X$ in $S_1(y) \backslash \{ x \}$. The set $Y$ is nonempty due to the girth 4 assumption.
  Then we have the following double-counting of
  the edges between $X$ and $Y$:
  \begin{equation} \label{eq:doublcountedges}
  \sum_{w \in X} d_w^H = | E(X,Y) | \le \sum_{z \in Y} d_z^H,
  \end{equation}
  where $d_w^H$ is the vertex degree of $w$ in $H$. Using
  distance-regularity, we obtain $d_w^H = d_z^H = c_2-1$ and
  \eqref{eq:doublcountedges} implies $|X| \le |Y|$. We can now apply
  Hall's Marriage Theorem to conclude that there is a perfect matching
  between $S_1(x) \backslash \{ y \}$ and $S_1(y) \backslash \{ x \}$.

  \medskip
  
  By Theorem \ref{thm:kLLYk0comp}, \eqref{eq:k_LLY-distreg} implies
  $\kappa_0(x,y) \ge 0$. Combining this with Proposition
  \ref{prop:uppbdcurv}(i), we conclude $\kappa_0(x,y) = 0$.

  \medskip
  
  For the calculation of the Bakry-\'Emery curvature we employ the
  method presented at the beginning of Section 8 of \cite{CLP} and the
  notation introduced there. In view of Theorem 8.1(i) in \cite{CLP},
  we only need to verify that
  $\lambda_1 = \lambda_1(\Delta_{S_1''(x)}) \ge \frac{d}{2}$, since
  then
  $$ \K_\infty(x) = \frac{3+d-{\rm av}^+_1(x)}{2} = \frac{3+d-(d-1)}{2} = 2. $$

  \medskip
  
  Triangle-freeness of $G$ implies
  $$ \Delta_{S_1''(x)} = \Delta_{S_1(x)} + \Delta_{S_1'(x)} =
  \Delta_{S_1'(x)}, $$
  where $\Delta_{S_1'(x)}$ is the weighted Laplacian on the $1$-sphere
  $S_1(x)$ with the following weights:
  $$ \text{$w_{y_1y_2}' = \sum_{\substack{z \in S_2(x)\\ y_1 \sim z \sim y_2}} \frac{1}{d_x^-(z)}$
  for all $y_1,y_2 \in S_1(x), y_1 \neq y_2$.} $$
  Since $G$ is distance-regular, we obtain $d_x^-(z) = c_2$ and
  $$ | \{ z \in S_2(x): y_1 \sim z \sim y_2 \} | = c_2 - 1. $$
  This implies $w_{y_1y_2}' = \frac{c_2-1}{c_2}$ and, therefore, the
  Laplacian $\Delta_{S_1'(x)}$ is $\frac{c_2-1}{c_2} \Delta_{K_d}$, where
  $\Delta_{K_d}$ is the non-normalized Laplacian of the complete
  graph $K_d$. Consequently, we have
  $$ \lambda_1(\Delta_{S_1'(x)}) = \frac{c_2-1}{c_2} \lambda_1(\Delta_{K_d})
  = \frac{c_2-1}{c_2} d \ge \frac{d}{2}, $$
  since $c_2 \ge 2$ because $G$ has girth $4$.
\end{proof}

It is tempting to assume that distance-regular graphs of girth $4$ are
always $(R)$-Ricci flat and then using Theorems \ref{thm:ORcurvRF}
and \ref{thm:BEcurvRF}(b) to conclude the statement of Theorem
\ref{thm:distreg}. However, the following example shows that this
assumption is {\bf{not}} always true. It remains an open question, however,
whether every distance-regular graph of girth $4$ is Ricci flat. 

\begin{example}[Incidence graph of $(11,6,3)$-design] \label{ex:incidence_graph}
	
  This is a distance-regular graph with intersection array
  $\{6,5,3;1,3,6\}$ (see \cite{Distreg}).
	
  The structure of the incomplete $2$-ball around a vertex $x$
  is given by:
	\begin{eqnarray*}
	S_1(x)& = & \{v_1,...,v_6\} \quad \text{and } \quad S_2(x)=\{v_7,...,v_{16}\}\\
	v_1 & \sim & v_8,v_{11},v_{13},v_{14},v_{15}\\
	v_2 & \sim & v_7,v_{10},v_{11},v_{12},v_{13}\\
	v_3 & \sim & v_9,v_{10},v_{11},v_{15},v_{16}\\
	v_4 & \sim & v_7,v_8,v_{10},v_{14},v_{16}\\
	v_5 & \sim & v_8,v_9,v_{12},v_{13},v_{16}\\
	v_6 & \sim & v_7,v_9,v_{12},v_{14},v_{15}
	\end{eqnarray*}

	We give an indirect prove that this graph is not $(R)$-Ricci flat. Assume otherwise, i.e., there exists an associated matrix $A$ with only $0$ entries on diagonal. The other possible entries of $A$ listed as below:
  
  $$ \left( \begin{array}{c|c|c|c|c|c|c}
  		 & 1 & 2 & 3 & 4 & 5 & 6 \\
  		 \hline
  		 1 & 0 & \textbf{\textcolor{red}{11}}, \textbf{\textcolor{blue}{13}} & \textcolor{blue}{11},\textcolor{red}{15} & \textcolor{red}{8},\textcolor{blue}{14} & \textcolor{blue}{8},\textcolor{red}{13} & \textcolor{red}{14},\textcolor{blue}{15} \\
  		 
  		 2 & \textcolor{blue}{11},\textcolor{red}{13} & 0 & \textcolor{blue}{10},\textcolor{red}{11} & \textcolor{blue}{7},\textcolor{red}{10} & \textcolor{red}{12},\textcolor{blue}{13} & \textcolor{red}{7},\textcolor{blue}{12} \\
  		 3 & \textcolor{red}{11},\textcolor{blue}{15} & \textcolor{red}{10},\textcolor{blue}{11} & 0 & \textcolor{blue}{10},\textcolor{red}{16} & \textcolor{red}{9},\textcolor{blue}{16} & \textcolor{blue}{9},\textcolor{red}{15} \\
  		 4 & 8,14 & 7,10 & 10,16 & 0 & \underline{8,16} & \underline{7,14} \\
  		 5 & 8,13 & 12,13 & 9,16 & 8,16 & 0 & 9,12 \\
  		 6 & 14,15 & 7,12 & 9,15 & 7,14 & 9,12 & 0   	
  		\end{array} \right), $$
  		
 $$ \left( \begin{array}{c|c|c|c|c|c}

 0 & \textbf{\textcolor{red}{11}}, \textbf{\textcolor{blue}{13}} & \textcolor{blue}{11},\textcolor{red}{15} & \textcolor{red}{8},\textcolor{blue}{14} & \textcolor{blue}{8},\textcolor{red}{13} & \textcolor{red}{14},\textcolor{blue}{15} \\
 
 \textcolor{blue}{11},\textcolor{red}{13} & 0 & \textcolor{blue}{10},\textcolor{red}{11} & \textcolor{blue}{7},\textcolor{red}{10} & \textcolor{red}{12},\textcolor{blue}{13} & \textcolor{red}{7},\textcolor{blue}{12} \\
 
 \textcolor{red}{11},\textcolor{blue}{15} & \textcolor{red}{10},\textcolor{blue}{11} & 0 & \textcolor{blue}{10},\textcolor{red}{16} & \textcolor{red}{9},\textcolor{blue}{16} & \textcolor{blue}{9},\textcolor{red}{15} \\
 
 8,14 & 7,10 & 10,16 & 0 & \underline{8,16} & \underline{7,14} \\
 8,13 & 12,13 & 9,16 & 8,16 & 0 & 9,12 \\
 14,15 & 7,12 & 9,15 & 7,14 & 9,12 & 0   	
 \end{array} \right), $$

Recall that the matrix $A$ cannot have repeated entries in any row and column.
If the entry of $A_{12}$ is chosen to be \textbf{\textcolor{red}{11}}, then all entries for the first three rows are uniquely determined as  the numbers in \textcolor{red}{red}. Then the entry of $A_{46}$ cannot be either $7$ or $14$, due to appearance of them in the sixth column. Contradiction!

Similarly, if the entry of $A_{12}$ is chosen to be \textbf{\textcolor{blue}{13}}, then all entries for the first three rows must be the numbers in \textcolor{blue}{blue}. Then the entry of $A_{45}$ cannot be either $8$ or $16$ due to the fifth column. Contradiction!

In conclusion, the Incidence graph of $(10,6,3)$-design is not
$(R)$-Ricci flat, even though it is triangle-free and has both maximum
possible Bakry-\'Emery curvature $\K_\infty(x) = 2$ and maximum possible
Olliver Ricci curvature $\kappa_{LLY}(x,y) = \frac{2}{d}$.

However, the vertices of this graph are Ricci flat via the following
matrix choice for $A$:
$$ \begin{pmatrix}
  11 & 0 & 15 & 8 & 13 & 14 \\
  13 & 12 & 11 & 10 & 0 & 7 \\
  0 & 11 & 10 & 16 & 9 & 15 \\
  14 & 10 & 16 & 7 & 8 & 0 \\
  8 & 13 & 9 & 0 & 16 & 12 \\
  15 & 7 & 0 & 14 & 12 & 9
  \end{pmatrix}.
$$
\end{example}

\section*{Appendix: The complete bipartite graphs $K_{d,d}$}

  We will show the following facts:
  \begin{itemize}
  \item[(1)] $K_{d,d}$ is $(R)$-Ricci flat for all $d$,
  \item[(2)] $K_{d,d}$ is $(S)$-Ricci flat for all $d$,
  \item[(3)] $K_{d,d}$ is $(RS$)-Ricci flat if and only if $d$ is even.
  \end{itemize}
  As before, we translate Ricci flatness properties at a vertex $x$, given
  by the maps $\eta_i$, into properties of the associated
  $d \times d$-matrix $A = (A_{ij})$. Since $K_{d,d}$ is triangle-free,
  we use a slightly different enumeration system for the matrix $A$:
  Let
  $S_1(x) = \{y_1,\dots, y_d\}$ where $y_j := \eta_j(x)$, and
  $S_2(x) =: \{z_1,\dots,z_t\}$ and, furthermore, $z_0 := x$.  Then the
  entries $A_{ij} \in \{0,1,\dots,t\}$ of $A$ are given via the
  relation
  $$ z_{A_{ij}} = \eta_i(y_j) $$
  and we have the following correspondences:
  \begin{itemize}
  \item[(a)] $\eta_i$ is injective corresponds to $A_{ij} \neq A_{ik}$ for
    all $j \neq k$,
  \item[(b)] $\eta_i(y_k) \neq \eta_j(y_k)$ corresponds to
    $A_{ik} \neq A_{jk}$ for all $i \neq j$,
  \item[(c)] $\eta_i^2(x) = x$ corresponds to $A_{ii} = 0$,
  \item[(d)] $\eta_j(\eta_ix) = \eta_i(\eta_jx)$ corresponds to $A_{ji}=A_{ij}$.
  \end{itemize}
  In other words, (a) corresponds to the property that $A$ has no
  repeated entries in the $i$-th row and (b) correspond to the
  property that $A$ has no repeated entries in the $k$-th
  column. Moreover, $(R)$-Ricci flatness requires in addition that the
  matrix $A$ has only the entry $0$ on the diagonal, $(S)$-Ricci flatness
  requires that $A$ is symmetric, and $(RS)$-Ricci flatness requires both
  additional properties of the matrix $A$. Note the
  general fact:
  \begin{itemize}
  \item[(e)] The number of occurrences of the entry
    $m \in \{0,\dots,t\}$ in the matrix $A$ is equal to $d_x^-(z_m)$.
  \end{itemize}  

  (1)-(3) can now be shown by providing suitable matrices $A$. 
  
  Proof of (1): 
  $$ A = \begin{pmatrix} 0 & 1 & 2 & \cdots & d-1 \\
    d-1 & 0 & 1 & \cdots & d-2 \\
    d-2 & d-1 & 0& \cdots & d-3 \\
    \vdots & \vdots & \vdots & \ddots & \vdots \\
    1 & 2 & 3 & \cdots & 0 \end{pmatrix}. $$
  Note that the first row of $A$ is fixed and the following rows are
  obtained by a right shift of the previous row.
  
  Proof of (2):
  $$ A = \begin{pmatrix} 0 & 1 & 2 & \cdots & d-1 \\
    1 & 2 & 3 & \cdots & 0 \\
    2 & 3 & 4 & \cdots & 1 \\
    \vdots & \vdots & \vdots & \ddots& \vdots \\
    d-1 & 0 & 1 & \cdots & d-2 \end{pmatrix}. $$
  Note that the first row of $A$ is fixed and the following rows are
  obtained by a left shift of the previous row.

  Proof of (3): Assume $d = 2n$ even. Then we can choose $A$ to be
  $$ {\scriptsize{\left( \begin{array}{cccccc|ccccc|c} 0 & 1 & 2 & 3 & \cdots & n-1 & n & n+1 & \cdots & 2n-3 & 2n-2 & 2n-1 \\
    1 & 0 & 3 & 4 & & n & n+1 & & & & 2n-1 & 2 \\
    2 & 3 & 0 & & & \udots & & & & \udots & 1 & 4 \\
    3 & 4 & & \ddots & & & & & & & 2 & 6 \\
    \vdots & & & & & & & \udots & & & & \vdots \\
    n-1 & n & \udots & & & 0 & 2n-1 & 1 & & & n-2 & 2n-2 \\
    \hline
    n & n+1 & & & & 2n-1 & 0 & 2 & & & n-1 & 1 \\
    n+1 & & & & \udots & 1 & 2 & \ddots & & & n & 3 \\
    \vdots & & & & & 2 & & & \udots & \udots & & \vdots \\
    \vdots & & \udots & & & & & & \udots & & & \vdots \\
    2n-2 & 2n-1 & 1 & 2 & & n-2 & n-1 & n & & & 0 & 2n-3 \\
    \hline          
    2n-1 & 2 & 4 & 6 & \cdots & 2n-2 & 1 & 3 & \cdots & \cdots & 2n-3 & 0
                         \end{array} \right),}} $$
  constructed as follows:
  \begin{itemize}
  \item $A_{ii}=0$ for all $1 \le i \le 2n$,
  \item $A_{ij}= i+j-2$ for $i \neq j$ and $i+j \le 2n+1$,
  \item $A_{ij}= i+j-2n-1$ for $i \neq j$, $i+j \ge 2n+2$ and $i,j \le 2n$,
  \item $A_{i,2n}=A_{2n,i} = 2(i-1)$ for $1 \le i \le n$,
  \item $A_{i,2n}=A_{2n,i} = 2(i-n)-1$ for $n+1 \le i \le 2n-1$. 
  \end{itemize}
  Finally, assume that $d$ is odd and $x$ is $(RS)$-Ricci flat with
  associated symmetric matrix $A$ with vanishing diagonal. Since
  $d_x^-(z_m) = d$ for all $m \in \{0,\dots,t\}$, each entry $m$
  appears exactly $d$ times in the matrix $A$ by (e) above. Since $d$
  is odd and $A$ symmetric, every entry must appear at least once on
  the diagonal, contradicting to the assumption of a vanishing
  diagonal.

\end{document}